\documentclass{article}
\usepackage[utf8]{inputenc}

\usepackage{amssymb}
\usepackage{amsthm}
\usepackage{amsmath,amscd}
\usepackage[mathscr]{euscript}
\usepackage[all]{xy}
\usepackage{lmodern}
\usepackage[T1]{fontenc}
\usepackage[textwidth=14cm,hcentering]{geometry}
\usepackage[colorlinks=true,linkcolor=red,citecolor=blue]{hyperref}
\usepackage{cleveref}
\usepackage{mathtools}
\usepackage{xspace}
\usepackage{tikz}
\usetikzlibrary{cd}
\usetikzlibrary{arrows,positioning}
\usepackage{bbm}

\usepackage{aliascnt}

\title{$K(1)$-local $K$-theory of Azumaya algebras}
\author{Maxime Ramzi }
\date{}

\newtheorem{thm}{Theorem}[section]
\newaliascnt{lm}{thm}  
\newtheorem{lm}[lm]{Lemma}
\aliascntresetthe{lm}
\Crefname{lm}{Lemma}{Lemmas}
\newaliascnt{prop}{thm}  
\newtheorem{prop}[prop]{Proposition}
\aliascntresetthe{prop}
\Crefname{prop}{Proposition}{Propositions}
\newaliascnt{cor}{thm}  
\newtheorem{cor}[cor]{Corollary}
\aliascntresetthe{cor}
\Crefname{cor}{Corollary}{Corollaries}

\newtheorem*{thm*}{Theorem}
\newtheorem*{cor*}{Corollary}
\newtheorem{thmx}{Theorem}

\theoremstyle{definition}
\newaliascnt{defn}{thm}  
\newtheorem{defn}[defn]{Definition}
\aliascntresetthe{defn}
\Crefname{defn}{Definition}{Definitions}
\newaliascnt{cons}{thm}  
\newtheorem{cons}[cons]{Construction}
\aliascntresetthe{cons}
\Crefname{cons}{Construction}{Constructions}
\newaliascnt{nota}{thm}  
\newtheorem{nota}[nota]{Notation}
\aliascntresetthe{nota}
\Crefname{nota}{Notation}{Notations}
\newaliascnt{conv}{thm}  

\aliascntresetthe{conv}
\Crefname{conv}{Convention}{Conventions}
\newaliascnt{ex}{thm}  
\newtheorem{ex}[ex]{Example}
\aliascntresetthe{ex}
\Crefname{ex}{Example}{Examples}
\newaliascnt{rmk}{thm}  
\newtheorem{rmk}[rmk]{Remark}
\aliascntresetthe{rmk}
\Crefname{rmk}{Remark}{Remarks}
\newaliascnt{ques}{thm}  
\newtheorem{ques}[ques]{Question}
\aliascntresetthe{ques}
\Crefname{ques}{Question}{Questions}
\newaliascnt{conj}{thm}  
\newtheorem{conj}[conj]{Conjecture}
\aliascntresetthe{conj}
\Crefname{conj}{Conjecture}{Conjectures}
\newaliascnt{warn}{thm}  
\newtheorem{warn}[warn]{Warning}
\aliascntresetthe{warn}
\Crefname{warn}{Warning}{Warnings}
\newaliascnt{obs}{thm}  
\newtheorem{obs}[obs]{Observation}
\aliascntresetthe{obs}
\Crefname{obs}{Observation}{Observations}
\newtheorem*{ques*}{Question}
\newtheorem*{rmk*}{Remark}
\newtheorem*{ex*}{Example}
\newtheorem*{defn*}{Definition}
\newaliascnt{recoll}{thm}  

\aliascntresetthe{recoll}
\Crefname{recoll}{Recollection}{Recollections}

\AtBeginEnvironment{defn}{\pushQED{\qed}}
\AtEndEnvironment{defn}{\popQED\enddefn}
\AtBeginEnvironment{defn*}{\pushQED{\qed}}
\AtEndEnvironment{defn*}{\popQED\enddefn}
\AtBeginEnvironment{cons}{\pushQED{\qed}}
\AtEndEnvironment{cons}{\popQED\endcons}
\AtBeginEnvironment{assu}{\pushQED{\qed}}
\AtEndEnvironment{assu}{\popQED\endassu}
\AtBeginEnvironment{nota}{\pushQED{\qed}}
\AtEndEnvironment{nota}{\popQED\endnota}
\AtBeginEnvironment{nota*}{\pushQED{\qed}}
\AtEndEnvironment{nota*}{\popQED\endnota}
\AtBeginEnvironment{conv}{\pushQED{\qed}}
\AtEndEnvironment{conv}{\popQED\enddefn}\AtBeginEnvironment{ex}{\pushQED{\qed}}
\AtEndEnvironment{ex}{\popQED\endex}
\AtBeginEnvironment{exs}{\pushQED{\qed}}
\AtEndEnvironment{exs}{\popQED\endexs}
\AtBeginEnvironment{exc}{\pushQED{\qed}}
\AtEndEnvironment{exc}{\popQED\endexc}
\AtBeginEnvironment{rmk}{\pushQED{\qed}}
\AtEndEnvironment{rmk}{\popQED\endrmk}
\AtBeginEnvironment{rmk*}{\pushQED{\qed}}
\AtEndEnvironment{rmk*}{\popQED\endrmk}
\AtBeginEnvironment{ques}{\pushQED{\qed}}
\AtEndEnvironment{ques}{\popQED\endques}
\AtBeginEnvironment{ques*}{\pushQED{\qed}}
\AtEndEnvironment{ques*}{\popQED\endques}
\AtBeginEnvironment{conj}{\pushQED{\qed}}
\AtEndEnvironment{conj}{\popQED\endconj}
\AtBeginEnvironment{warn}{\pushQED{\qed}}
\AtEndEnvironment{warn}{\popQED\endwarn}
\AtBeginEnvironment{obs}{\pushQED{\qed}}
\AtEndEnvironment{obs}{\popQED\endobs}

\newcommand{\op}{^{\mathrm{op}}}
\newcommand{\cat}{\mathrm}
\newcommand{\Cat}{\cat{Cat}}
\newcommand{\Catperf}{\mathrm{Cat}^{\mathrm{perf}}}
\newcommand{\on}{\operatorname}
\newcommand{\id}{\mathrm{id}}
\newcommand{\Fun}{\on{Fun}}

\newcommand{\Map}{\on{Map}}

\newcommand{\Z}{\mathbb Z}
\newcommand{\Sph}{\mathbb S}

\newcommand{\Ss}{\mathscr{S}}
\newcommand{\Sp}{\cat{Sp}}

\renewcommand{\O}{\mathcal{O}}

\newcommand{\PrL}{\mathrm{Pr}^\mathrm{L} }

\newcommand{\Alg}{\mathrm{Alg}}
\newcommand{\CAlg}{\mathrm{CAlg}}

\newcommand{\Mod}{\cat{Mod}}

\newcommand{\Perf}{\mathrm{Perf}}

\newcommand{\Ind}{\mathrm{Ind}}
\newcommand{\End}{\mathrm{End}}
\newcommand{\Idem}{\mathrm{Idem}}

\newcommand{\colim}{\mathrm{colim}}

\newcommand{\Psh}{\cat{Psh}}

\newcommand{\KO}{\mathrm{KO}}
\newcommand{\KU}{\mathrm{KU}}

\newcommand{\Br}{\mathrm{Br}}
\newcommand{\Pic}{\mathrm{Pic}}
\newcommand{\BBr}{\mathbf{Br}}
\newcommand{\PPic}{\mathbf{Pic}}
\newcommand{\Gm}{\mathbb{G}_m}
\newcommand{\Gpic}{\mathbb{G}_{\mathrm{pic}}}
\newcommand{\et}{\mathrm{\'et}}

\newcommand{\Mot}{\mathrm{Mot}}
\newcommand{\Sh}{\mathrm{Sh}}
\newcommand{\udl}[1]{\underline{#1}}
\newcommand{\QCoh}{\mathrm{QCoh}}
\newcommand{\Aut}{\mathrm{Aut}}
\newcommand{\U}{\mathcal{U}}
\newcommand{\Spec}{\mathrm{Spec}}
\newcommand{\one}{\mathbbm{1}}
\newcommand{\V}{\mathcal V}

\newcommand{\st}{\mathrm{st}}

\newcommand{\coker}{\mathrm{coker}}

\newcommand{\Gal}{\mathrm{Gal}}

\begin{document}

\maketitle
\begin{abstract}
We compute certain strict Picard spectra of $K(1)$-local $K$-theory spectra of schemes in terms of Brauer groups, using the map that takes an Azumaya algebra to its $K(1)$-local $K$-theory and proving a Künneth formula in that setting. For example, we prove that for semi-local rings of characteristic $\neq p$, $\BBr(R)[p^\infty]\simeq \Gpic(L_{K(1)}K(R)\otimes\Sph_{W(\overline{\mathbb F}_p)})[p^\infty]$, where $\Gpic$ is Carmeli's strict Picard spectrum. We prove the same result for $R[\frac{1}{p}]$, when $R$ is $p$-henselian.   
\end{abstract}
\tableofcontents
\newpage
\section*{Introduction}
\subsection*{Background}
The group of units of a commutative ring $R$ is typically very different from its additive group. This difference becomes more stringent when one allows $R$ to be a commutative ring \emph{spectrum}, where the underlying spectrum of $R$ and its spectrum of units $\mathfrak{gl}_1(R)$ are of vastly different nature. To see an example of this, one need look no further than the initial commutative ring spectrum, the sphere spectrum $\Sph$, where $\Sph$ receives no interesting map from (connective) complex topological $K$-theory, $\mathrm{ku}$, while $\mathfrak{gl}_1(\Sph)$ receives (up to a shift) the $J$-homomorphism, a very important map in homotopy theory. 

One is therefore naturally lead to the study of \emph{spectra} of units of commutative ring spectra. One convenient way to approximate this study is to sudy instead the spectrum of \emph{strict units} of $R$: 
\begin{defn*}
    The connective spectrum of strict units of a commutative ring spectrum $R$ is defined as the following connective mapping spectrum: $$\Gm(R):=\Map_{\Sp^{\mathrm{cn}}}(\mathbb Z, \mathrm{gl}_1(R))$$
\end{defn*}
This object, and variants thereof, is under increasing scrutiny, see e.g. \cite{Shacharstrict}, \cite{ChroNS} and \cite{shacharkiran}. Among other things, it partly governs ``flat algebraic geometry'' over $R$, since it is equivalent to maps of commutative $R$-algebras from the flat Laurent polynomials ring $R[t^{\pm 1}]$ to $R$. 

In this paper, we study the spectrum of strict units of a large class of height $1$ commutative ring spectra, namely the $K(1)$-local $K$-theory of (classical) schemes. Work of Thomason \cite{thomason1985algebraic} (see also \cite{CM}) shows that this $K(1)$-local $K$-theory is closely related to étale cohomology, and this is perhaps the main source of examples of commutative ring spectra of height $1$.

Since $K$-theory is some kind of decategorification process, a natural guess is that the (strict) units of the ($K(1)$-local) $K$-theory of a scheme $X$ has something to do with the Picard group of that scheme, and we will prove that there is in fact a tight relationship between the two - at the level of $\pi_0$, this relationship is very classical, for example for Dedking domains where $K_0(R)\cong \mathbb Z \oplus \Pic(R)$.

More generally, the study of $\Gm$ itself leads to the study of a canonical delooping thereof, namely the spectrum of strict picard elements, $\Gpic$, defined similarly, see \cite{Shacharstrict}. In that case, the ``decategorification philosophy'' mentioned above suggests that the (strict) picard elements of the ($K(1)$-local) $K$-theory of a scheme $X$ has something to do with the \emph{Brauer group} of that scheme, and it is in fact this relationship that most of this paper is devoted to. 




To state our precise results, we start with a (too good to be true!) dream and add some symbols in order to, slowly but surely, arrive at a correct statement. 

Since (derived\footnote{Our Brauer groups and Azumaya algebras are implicitly derived throughout.}) Brauer classes over $X$ are exactly invertible objects in $\Catperf_X$, the $\infty$-category of $X$-linear stable $\infty$-categories, also known as quasi-coherent stacks on $X$, one is led to hope that taking their $K$-theory would yield invertible modules over $K(X)$, establishing the aforementioned connection between the Brauer group of $X$ and the Picard group of the commutative ring spectrum $K(X)$. This would look like a map $$``\Br(X)\to \Pic(K(X))"$$
If the world were perfect, this group morphism would be an isomorphism, or not far from it, up to understandable phenomena. 

Unfortunately, none of this works. For Azumaya algebras $A,A'$, even over a field, one typically does not have $K(A)\otimes_{K(X)}K(A')\simeq K(A\otimes_{\O_X}A')$\footnote{In fact, this is almost never the case - if $K(X)$ is connective, having this property for $A$ and $A\op$ implies that $A$ is trivial in the Brauer group, cf. \cite[Corollary 2.22]{companionMot} }. So we do not even have a group morphism, let alone a near-isomorphism! There are however at least two ways to make a Künneth-type formula like this one true - we start by describing the first one. If the prime $p$ is invertible in the base $X$, $K(1)$-localization tends to make Künneth formulas ``more true''. For example we will see that when $X$ is the spectrum of a semi-local ring\footnote{Or the $p$-inverted $p$-henselization of a commutative ring.}, the above formula always holds for Azumaya algebras which are $p$-primary torsion in the Brauer group. If this is true, the above construction gives a morphism from the $p$-primary torsion part of the Brauer group of $X$ to the ($K(1)$-local) Picard group of $L_{K(1)}K(X)$, variants of which have been studied a fair amount. One may then again hope that the morphism $$\Br(X)[p^\infty]\to \Pic(L_{K(1)}K(X))$$
is injective. This fails for an even sillier reason: if the Künneth formula holds, then since $A\op$ is Morita inverse to $A$ and $K(A\op)\simeq K(A)$ as $K(X)$-modules, $L_{K(1)}K(A)$ is forced to be a $2$-torsion Picard element! For example, if $p$ is odd, this forces it to be trivial: our morphism is not interesting. 
\begin{rmk*}
    When $p=2$, we \emph{can} actually get interesing Picard elements this way. For example, over $\mathbb R$ and at the prime $2$, the quaternion algebra $\mathbb H$ satisfies $L_{K(1)}K(\mathbb H)\simeq \Sigma^4 \KO_2$ which is the nontrivial $2$-torsion element in $$\Pic(L_{K(1)}K(\mathbb R))= \Pic(\KO_2)\cong \Z/8.$$ 
\end{rmk*}
Note, however, that we used the \emph{property} that $A$ is $p$-primary torsion in the Brauer group instead of remembering that as extra data. Indeed, the Brauer group of a scheme is better viewed as $\pi_0$ of a more natural Brauer space or spectrum $\BBr(X)$, in which there is room for homotopies. Similarly, the Picard group $\Pic(L_{K(1)}K(X))$ is $\pi_0$ of the more natural Picard space or spectrum $\PPic(L_{K(1)}K(X))$ (which is used to define the strict Picard spectrum!). For a connective spectrum $E$, we let $E[p^\infty]$ denote the (connective, homotopy) fiber of the map $E\to E[\frac{1}{p}]$ - this consists of points in $E$ equipped with a nullhomotopy of one of their $p$-power powers. So, at least when $X$ is the spectrum of a semi-local ring, we actually obtain a map $$\BBr(X)[p^\infty]\to\PPic(L_{K(1))}K(X))[p^\infty]$$

We remark that in \cite{Toen}, Toën shows that $\pi_0(\BBr(X)[p^\infty])$ is given by the étale cohomology group $H^2_\et(X,\mu_{p^\infty})$ up to a summand $H^0_\et(X,\Z/p^\infty)$ with little geometric meaning\footnote{This summand is there to observe that $\O_X$ is itself $p^n$-torsion in many ways, namely by suspensions. In particular it does not have much to do with Brauer groups, so on top of it making some statements false, it is mostly irrelevant.}. In fact, the whole space can be described as $\BBr(X)[p^\infty]\simeq \Gamma_\et(X,B^2\mu_{p^\infty})\times H^0_\et(X,\Z/p^\infty)$. We will \emph{ignore} this second factor, and for the purposes of this paper, use the following notation: 
\begin{nota}\label{conv:Br}
   Instead of the usual derived Brauer sheaf which is equivalent to $B^2\Gm\times B\Z$, we let $\BBr(-):= B^2\Gm$, so $\BBr(X) = \Gamma_\et(X,B^2\Gm)$. 
   
   In particular, 
    
    $\BBr(-)[p^\infty]= B^2\mu_{p^\infty} =\Gamma_\et(-,B^2\mu_{p^\infty})$ (this is a subspace of $E[p^\infty]$, for $E$ the actual derived Brauer space).
\end{nota}
With this notation, one thing we will prove is that the above map is in fact injective on $\pi_0$ (and a variant with a smaller source will be true for more general $X$). 

However, we should not stop here : since $\mu_{p^\infty}$ is a sheaf of ordinary abelian groups, this cohomology space has more structure than a mere spectrum: all its elements are \emph{strict}. More precisely, it is a $\Z$-module spectrum (equivalently, can be considered as an object of the derived category $D(\Z)$). Thus, our map actually admits a refinement to $\Gpic := \Map_{\Sp^{\mathrm{cn}}}(\Z,\PPic)$, the strict Picard spectrum mentioned earlier. That is, we find a map $\BBr(X)[p^\infty]\to \Gpic(L_{K(1)}K(X))[p^\infty]$. We have now put a considerable more amount of structure on $L_{K(1)}K(A)$ and are allowed to hope for an isomorphism statement. It turns out that this is not quite right, as one can see by looking at algebraically closed fields: the issue is that the $K(1)$-local $K$-theory of algebraically closed fields is not itself exactly algebraically closed: it is $p$-complete topological $K$-theory, as opposed to the Lubin--Tate theory associated to the formal group $(\overline{\mathbb F}_p,\widehat{\Gm})$. One can fix this by tensoring our commutative ring spectra with the spherical Witt vectors $\Sph_{W(\overline{\mathbb F}_p)}$. 

Our first main theorem is that in general, all of this works after looping once, but over a semi-local ring or the $p$-inversion of a $p$-henselian ring, this whole story actually pans out: 
\begin{thmx}[{\Cref{thm:generalblahG}, \Cref{cor:perfectoid}, \Cref{cor:semilocal}}]\label{thm:BrwithK}
    Let $X$ be a qcqs scheme such that $p\in\Gm(X)$\footnote{Here and throughout this introduction, $p$ denotes a fixed prime number. $K(1)$- and $T(n)$-localizations are to be understood with respect to this implicit prime.}. The canonical map $$\PPic(X)[p^\infty]\to \Gm(L_{K(1)}K(X)\otimes\Sph_{W(\overline{\mathbb F}_p)})[p^\infty]$$ is an equivalence. 

    If $X$ is the spectrum of a semi-local ring, or $X$ is the spectrum of $R[\frac{1}{p}]$ for some $p$-henselian commutative ring $R$, this map deloops to an equivalence $$\BBr(X)[p^\infty]\to \Gpic(L_{K(1)}K(X)\otimes\Sph_{W(\overline{\mathbb F}_p)})[p^\infty]$$
    which is given by $A\mapsto L_{K(1)}K(A)$. More generally, there is a subspace $\widetilde{\BBr}(X)[p^\infty]\subset~\BBr(X)[p^\infty]$, to be defined later\footnote{Which, for a semi-local ring, or a ring of the form $R[\frac{1}{p}]$ with $R$ $p$-henselian, agrees with $\BBr(X)[p^\infty]$}, and a map described as above which is an inclusion of components $$\widetilde{\BBr}(X)[p^\infty]\to\Gpic(L_{K(1)}K(X)\otimes\Sph_{W(\overline{\mathbb F}_p)})[p^\infty] $$ 
\end{thmx}
\begin{rmk}\label{rmk:ZactSphW}
    Here, $\Sph_{W(\overline{\mathbb F}_p)}$ denotes the spherical Witt vectors of the field $\overline{\mathbb F}_p$, see e.g. \cite[Section 2.1]{ChroNS} or Lurie's original account \cite[Section 5.2]{Ell2}. As explained above, they act here as a sort of ``renormalization'' to ``fix'' the fact that $L_{K(1)}K(\mathbb C)$ is not ``algebraically closed''.

    Note that this $\mathbb E_\infty$-ring spectrum has a ($p$-completely) locally unipotent $\Z$-action by a spherical Witt lift of the Frobenius automorphism, which induces a descent equivalence between $p$-complete modules over it with a unipotent Frobenius semi-linear automorphism and $p$-complete spectra. In particular, if one wants to remove it, one has to take $\Z$-fixed points of both sides of our map. However, by design, the action on the source of our map is trivial, so that one finds that without the spherical Witt vectors term, the $\Gpic$ in question (say in the case of a local ring) is simply $$\BBr(X)[p^\infty]^{S^1}\simeq \BBr(X)[p^\infty]\oplus\PPic(X)[p^\infty].$$ 
\end{rmk}

\begin{rmk*}
    We will soon see that $\Gpic(L_{K(1)}K(X)\otimes\Sph_{W(\overline{\mathbb F}_p)})[p^\infty]$ is always \emph{some} subspace of $\BBr(X)[p^\infty]$, equivariantly with respect to the unipotent $\Z$-action described above. 
\end{rmk*}
\begin{rmk*}
    The assumption that $p\in\Gm(X)$ in the above theorem is in some sense unavoidable: $K(1)$-local $K$-theory is insensitive to inverting $p$, see \cite[Theorem 1.1]{BCM}. 
\end{rmk*}
We provide some examples of applications of this theorem. First, we use it to equip $\Sigma^4\KO_2$ with a strict $2$-torsion structure by observing that $L_{K(1)}K(\mathbb H)\simeq \Sigma^4\KO_2$, recovering an unpublished calculation of Antieau. We also deduce from this the \emph{non-existence} of a spectral lift of the ``Clifford algebra'' morphism from the Witt group of $\mathbb R$ to its Brauer--Wall group: 
\begin{cor*}[{\Cref{cor:nofact}}]
    There exists no map of spectra $L(\mathbb R)\to \mathbf{BW}(\mathbb R)$ which, on $\pi_0$, recovers the assignment $(V,q)\mapsto \mathrm{Cl}(V,q)$. 
\end{cor*}
Here, $\mathbf{BW}(\mathbb R)$ is the Brauer--Wall spectrum, which is to the Brauer--Wall group what the Brauer spectrum is to the Brauer group.

Second, we use our result to compute $\Gm(K(R))[p^\infty]$ for various rings $R$, such as $p$-inverted rings of integers in number fields. In that context, we have (see \Cref{section:numberring} for a slightly more general statement):
\begin{cor*}[\Cref{cor:GmLQ}]
     Let $R=\O_F[\frac{1}{p}]$ or some further localization thereof, where $F$ is a number field. 

     In this case, there is an equivalence $$\Gm(K(R))[p^\infty]\simeq \PPic(R)[p^\infty]\oplus \Gm(R)[p^\infty]$$
\end{cor*}
 As already mentioned, part of \Cref{thm:BrwithK} in the case of a field is the statement that $p$-power torsion Brauer classes satisfy a Künneth theorem. In full generality, we are unable to prove this, but we have the following result:
\begin{thmx}[\Cref{cor:AzKünneth}]\label{thm:Kintro}
    Let $X$ be a qcqs scheme such that $p\in\Gm(X)$. Suppose $A$ is an Azumaya algebra over $X$ such that after pullback to $X\times_{\Spec(\Z[\frac{1}{p}])}\Spec(\Z[\frac{1}{p},\zeta_{p^\infty}])$, $A$ is split by a $G$-Galois cover, for some $p$-group $G$. 

    For any $X$-linear category $C$, the Künneth map $$K(A)\otimes_{K(X)}K(C)\to K(\Perf(A)\otimes_{\Perf(X)}C)$$ is a $K(1)$-local equivalence. 
\end{thmx}
The subspace spanned by those Azumaya algebras satisfying this result is what we called $\widetilde{\BBr}(X)[p^\infty]$. 
\begin{ex*}
    By the Merkurjev-Suslin theorem \cite[Corollary, pp.158]{MS}, resp. its extension to semilocal rings \cite{hoobler}, the hypotheses of the above theorem are satisfied for any $p$-power torsion Brauer class over a field, resp. over a semilocal ring.
    \end{ex*}
\begin{ex*}
    Another source of examples comes from the theory of perfectoid rings, cf. \Cref{cor:perfectoid}: whenever $R$ is $p$-henselian, Azumaya algebras over $R[\frac{1}{p}]$ also satisfy the assumptions of the above theorem.
\end{ex*}
Removing the symbol ``$\mathbb G$'' but keeping the torsion part, we lose some of the tools needed to establish \Cref{thm:BrwithK}. On the level of $\PPic(X)$, though, we nonetheless get the following injectivity result:
\begin{thmx}[\Cref{thm:nostricttext}]\label{thm:nostrict}
    Let $X$ be a qcqs scheme such that $p\in\Gm(X)$. The canonical map $$\PPic(X)[p^\infty]\to GL_1(L_{K(1)}K(X))[p^\infty]$$ is injective on $\pi_0$. 
\end{thmx}
\begin{rmk*}
   We note that in this theorem, the proof works better \emph{without} the spherical Witt vectors of \Cref{thm:BrwithK}.
\end{rmk*}
Now, if one wants a result that goes beyond the examples provided above (or more generally beyond the hypotheses of \Cref{thm:Kintro}) in \Cref{thm:BrwithK}, one should either try to generalize \Cref{thm:Kintro} (which the author has not succeeded in doing), or give up on having the nice concrete description ``$A\mapsto L_{K(1)}K(A)$''. In particular, a second way of enforcing Künneth type formulas is to work étale locally: Azumaya algebras are étale locally trivial, so that they trivially satisfy a ``local Künneth formula''. 

Thus, in a slightly less interesting way, because it does not include a Künneth type statement but only a verification on stalks, we also obtain the following: 
\begin{thmx}[\Cref{thm:Sh}]\label{thm:BrPicSh}
    Let $X$ be a qcqs scheme such that $p\in\Gm(X)$. The canonical map $$\BBr(X)[p^\infty]\to \Gpic(\Mod_{\mathcal K}(\Sh(X_\et, \Mod_{\Sph_{W(\overline{\mathbb F}_p)}}(\Sp_{K(1)}))))[p^\infty]$$ is an equivalence, where $\mathcal K$ is the étale sheaf given by $K(1)$-local $K$-theory. 
\end{thmx}
We note that unlike in the previous discussion, here $K(1)$-localizing was simply not needed to enforce the desired Künneth type statement, but it is needed to have the correct local picture.

While this result is less interesting because of the trivial enforcing of the Künneth formula, this construction lends itself better to an analysis of the local behaviour of the map $B^2\mu_{p^\infty}\to GL_1(L_{K(1)}K(-))[p^\infty]$. We thus obtain the following categorified version of \Cref{thm:nostrict}: 
\begin{thmx}[\Cref{thm:nostrictBrtext}]\label{thm:nostrictBr}
    Let $X$ be a qcqs scheme such that $p\in\Gm(X)$. The canonical map $$\BBr(X)[p^\infty]\to \PPic(\Mod_{\mathcal K}(\Sh(X_\et, \Sp_{K(1)})))[p^\infty]$$ is injective on $\pi_0$. In particular, the restricted map $$\widetilde{\BBr}(X)[p^\infty]\to \PPic(L_{K(1)}K(X))[p^\infty]$$ is also injective on $\pi_0$.
\end{thmx}

To prove these statements, besides the story we have told so far, we rely on the key calculation of $\Gpic(\KU_p\otimes\otimes\Sph_{W(\overline{\mathbb F}_p)}) $ done in \cite{ChroNS} which corresponds to the special case of our results for separably closed fields. There, we see that taking $[p^\infty]$ yields exactly $\Sigma^2\Z/p^\infty$, and our results stem from the observation that this $\Z/p^\infty$ is actually $\mu_{p^\infty}$, and it is ``the same'' $\mu_{p^\infty}$ as the one appearing in the description of $\BBr[p^\infty]$.

A second key technical input we need is a way to access the stalks of the étale presheaf $U\mapsto \Gpic(L_{K(1)}K(U)\otimes\Sph_{W(\overline{\mathbb F}_p)})$, which we believe to be of independent interest. We also believe that a stronger statement can be made, but for now we make a weaker one for which we need a definition:
\begin{defn*}
    Let $F: C\to D$ be a functor between categories with filtered colimits. We say $F$ is \emph{subfinitary} is for every filtered colimit in $C$, the induced diagram $$\colim_I F(c_i)\to F(\colim_I c_i)$$ is a monomorphism. 
\end{defn*}
Our technical input is: 
\begin{thmx}[\Cref{thm:finitary}]\label{thm:subfinintro}
    Fix a height $n\geq 0$. For any torsion abelian group $H$, The functor $$\CAlg(\Sp_{T(n)})\to D_{\geq 0}(\Z), R\mapsto \Map_\Z(\Sigma^{n-1} H,\Gpic(R))\simeq \Omega^{n-1}\Map_\Z(H,\Gpic(R))$$ is subfinitary. 
\end{thmx}
\begin{rmk*}
Robert Burklund mentioned to the author that together with Schlank and Yuan, they proved (independently) a stronger version of this theorem, namely they prove finitary-ness instead of subfinitary-ness (that is, commutation with filtered colimits) for the whole of $\Map_\Z(H,\Gpic(R))$ and not its $(n-1)$-fold loop space. Their proof is more sophisticated and ours is more elementary, but only yields the above weaker result. This will however be sufficient for our purposes. 
\end{rmk*}
Finally, besides the value of each $\Gpic(E_n)$ from \cite{ChroNS} we will also need Ben-Moshe, Carmeli, Schlank and Yanovski's cyclotomic redshift \cite{Cyclored} to prove a stalk-version of our results that includes the verification that the map is the correct map: 
\begin{thmx}[\Cref{thm:Picshifttext}]\label{thm:Picshift}
Fix $n\geq 0$. 
    Let $E_n$ be any Lubin-Tate theory based on an algebraically closed field of characteristic $p$ if $n\geq 1$, or a separably closed field of characteristic $\neq p$ if $n=0$. For any map of commutative ring spectra $L_{K(n+1)}K(E_n)\to E_{n+1}$ to a height $(n+1)$ Lubin-Tate theory, also based on an algebraically closed field of characteristic $p$, the induced map $$\Gpic(E_n)[p^\infty]\to \Gm(L_{K(n+1)}K(E_n))[p^\infty]\to \Gm(E_{n+1})[p^\infty]$$  is an equivalence\footnote{Note that for $n=0$, $\Gpic(F)[p^\infty]\simeq \PPic(F)[p^\infty]$ for a separably closed field $F$, so in that case the theorem can be stated with strict or not strict Picard elements.}.
\end{thmx}
\begin{rmk*}
    Taking the results of \cite{ChroNS} and \cite{Cyclored} for granted, the above theorem might be more accurately called an observation, cf. also its proof. We label it as a theorem given how central it is to our other developments, though as such it should be attributed to the authors of these works.
\end{rmk*}
\begin{rmk*}
In the case $n=0$, we have the remarkable feature that up to a mild $\Sph_{W(\overline{\mathbb F}_p)}$-factor, $L_{K(1)}K(F)$ is essentially already\footnote{The reader will excuse our abuse of notation here.}
 $E_1(\overline{\mathbb F}_p)$

 This obviously raises the following questions for general $n$: first, what is the term in the middle, and can it be described in terms of the target by descent; and second, does the above map deloop to an equivalence between a strict Brauer group of $E_n$ and $\Gpic(E_{n+1})$ ?
\end{rmk*}
\begin{rmk*}
    The above can equivalently be phrased with $p$-completion, or even $p$-localization instead of $[p^\infty]$, and would perhaps match the statement of \cite[Theorem H]{ChroNS} better this way. But with $[p^\infty]$ it fits better with our Brauer group story, and also makes some of the proofs easier since the key points in the proof of \textit{loc. cit.} are rather proved using the $[p^\infty]$ perspective. 
\end{rmk*}

Besides these technical inputs (that constitute the heart of this paper), our method of proof is similar to \cite{companionMot}, namely, we identify the local behaviour of the invariants in question, and use the following trick: let $\mathcal{F\to G}$ is a map of truncated étale presheaves of spectra on $X$ such that $\mathcal F$ is actually a sheaf. If the map induces an equivalence on stalks $\mathcal F_{\overline x}\to \mathcal G_{\overline{x}}$, then it admits a retraction. Indeed, it then follows that the map $\mathcal F\to \mathcal G\to \mathcal G^a$ to the sheafification of $\mathcal G$ is an equivalence.
\subsection*{Related works}
As already mentioned, our basic insight here is the same as in the companion paper \cite{companionMot}, which is however easier on the technical side - the reader is advised to view the former as some kind of light introduction to the present paper. 

The idea to study strict units and strict Picard groups of ring spectra goes back at least to Hopkins and Lurie's unpublished calculation of the strict units of $E$-theory; but was subsequently taken up by Carmeli in the context of ring spectra ``like the sphere spectrum'' \cite{Shacharstrict} and Burklund--Schlank--Yuan as a consequence of their chromatic Nullstellensatz in the context of $E$-theories \cite{ChroNS}. More recently, Carmeli--Luecke also studied strict unit groups of $K$-theory ring spectra \cite{shacharkiran}, and we will explain how to recover their calculation of $\Gm(K(\mathbb F_q))$ using our methods.

Our \Cref{thm:BrPicSh} also has a precursor in the work of Moulinos \cite{moulinos}, specifically he already constructs there the map that induces our equivalences upon taking strict $p$-power torsion elements at the level of étale sheaves. We $K(1)$-localize to impose étale descent but also to get the ``correct'' stalks. 

In the direction of the tools we use, Burklund mentioned to us that with Schlank and Yuan, they would prove a stronger version of \Cref{thm:finitary} in forthcoming work, with different applications in mind. 

Finally, as in \cite{companionMot}, we follow here the investigations of Tabuada and Van den Bergh \cite{TvdB1,TvdB2,Tab0,Tab1,Tab2} and this paper is greatly inspired by the questions raised there. 

\subsection*{Sectionwise outline}
In \Cref{section:strictE} we prove \Cref{thm:Picshift} by putting together results from \cite{ChroNS} and \cite{Cyclored}. 

In \Cref{section:Galois}, we use the chromatic Fourier transform to relate Galois theory and strict Picard spectra, and establish basic properties of Galois extensions in ``completed'' settings to prove \Cref{thm:subfinintro}. 

\Cref{section:main} is the ``main'' section of this article, where we prove the Künneth formula \Cref{thm:Kintro} and combine it with the previous results to obtain our ``Brauer results''. We first prove the results about strict structures that involve the Künneth formula that add up to \Cref{thm:BrwithK}, then move on to the sheafy results still involving strict structures, that is, \Cref{thm:BrPicSh}, and then finally forget about strictness altogether and prove \Cref{thm:nostrict} and \Cref{thm:nostrictBr}. 


Finally, in \Cref{section:ex} we study specific examples, namely the quaternion algebra $\mathbb H$ at the prime $2$ and other Cliffoed algebras over the reals, as well as number rings. 
\subsection*{Conventions}
We use the language of $\infty$-categories as extensively developed in \cite{HTT,HA}, though we drop the ``$\infty$-'' and simply refer to them as categories, and say ``$1$-category'' for ordinary categories when the need arises. In particular we use the following standard notation: 
\begin{itemize}
    \item $\CAlg$ and $\Alg_{\mathbb E_n}$ for commutative and $\mathbb E_n$-algebras respectively, in a symmetric monoidal ($\infty$-)category; 
    \item $\PrL$ for the category of presentable categories, $\PrL_\st$ for its stable variant and $\PrL_\omega$ for the category of compactly generated presentable categories and compact-preserving functors between them; 
    \item $\Ss$ for the category of spaces/anima/$\infty$-groupoids, ...;
    \item We use $\Map$ for mapping spaces; when the relevant category is stable or additive, we often also use its automatic commutative group/connective spectrum structure\footnote{In this article we do not use nonconnective mapping spectra.}.
\end{itemize}
We also use the following notation: 
\begin{itemize}
    \item For a (symmetric) monoidal category $C$, we let $\one_C$ denote its unit, $\one$ if $C$ is clear from context; 
    \item If $C$ as above is also cocomplete, we use $\one_C[-]$ to denote the unique colimit-preserving symmetric monoidal functor $\Ss\to C$, so that $\one_C[X]\simeq \colim_X \one_C$ and if $X$ has some operadic structure, so does $\one_C[X]$; 
    \item Dually, and if there is a risk for confusion\footnote{We are specifically worried about $F(G,\one_C)$, where the notation $\one_C^G$ could lead to confusion with fixed points.}, we use $F(-,\one_C)$ to denote the mapping object from a space into $\one_C$; 
    \item For a symmetric monoidal category $C$, we let $\PPic(C)$ denote the connective Picard spectrum of $C$, with the following two abbreviations: for a qcqs scheme $X$, we let $\PPic(X):=~\PPic(\Perf(X))$ and for a $T(n)$-local commutative ring spectrum $R$, $\PPic(R) := ~\PPic(\Mod_R(\Sp_{T(n)}))$; 
    \item For $C$ as above, we let $\Gpic(C):=\Map_{\Sp^\mathrm{cn}}(\Z,\PPic(C))$ as in \cite{Shacharstrict}, and $\Gm(C)= ~\Omega\Gpic(C)$, with similar abbreviations as above for $\Gpic(R), \Gm(R),\Gm(X)$.
    \item For a qcqs scheme $X$, we let $\Catperf_X:= \Mod_{\Perf(X)}(\Catperf)$ and call its objects ``$X$-linear categories''. It is equipped with the relative Lurie tensor product $\otimes_{\Perf(X)}$.
\end{itemize}
Our Azumaya algebras and quasicoherent sheaves are typically implicitly derived, so that $\QCoh,\PPic$ etc. by default refer to the derived versions of quasicoherent sheaves and line bundles. However our schemes are classical, and qcqs.  We use standard notations such as $\mu_n(X), \Gm(X)$ for schemes, except for $\BBr(X)$, as indicated in \Cref{conv:Br}, namely we let $\BBr(X):=\Gamma_\et(X,B^2\Gm)$ and we consider it with its implicit map to the usual derived Brauer space which one can define as $\PPic(\Catperf_X)$ (cf. \cite{Toen}). For the Picard space and the Brauer space, we let the non-bold version denote $\pi_0$, that is, $\Pic:=\pi_0\PPic, \Br:=\pi_0\BBr$.

There is an implicit prime $p$ fixed throughout, and all chromatic localizations $T(n)$ are with respect to that implicit prime.
\subsection*{Acknowledgements}
I wish to thank Ben Antieau for helfpul comments and feedback about this work, as well as for letting me include some of his calculations in \Cref{ex:quaternion}. I'm also grateful to Tom Bachmann, Tess Bouis, Robert Burklund, Shachar Carmeli\footnote{Who should also be thanked for the inspiration behind much of this work.}, Elden Elmanto, Eugen Hellmann, Marc Hoyois, Shai Keidar, Achim Krause, Vova Sosnilo, Sebastian Wolf and Mura Yakerson for interesting discussions related to the subject matter. Special thanks to Dustin Clausen and Akhil Mathew for suggesting \Cref{cor:perfectoid}, and particularly to Akhil Mathew for allowing me to include his proof.

As for the companion paper, this work arose after an inspiring talk given by Tabuada in the conference ``Recent developments in algebraic $K$-theory'' organized at the University of Warwick, and I am grateful to the organizers, Rudradip Biswas and Marco Schlichting, for putting together this event.

This research was funded by the Deutsche Forschungsgemeinschaft (DFG, German Research Foundation) – Project-ID 427320536 – SFB 1442, as well as by Germany’s Excellence Strategy EXC 2044 390685587, Mathematics Münster: Dynamics–Geometry–Structure. 

\section{Strict unit and Picard spectra of Lubin Tate theories}\label{section:strictE}
The goal of this section is to prove \Cref{thm:Picshift}. We will mostly use it at heights $0$ and $1$ (in fact, for most of the paper, only at height $0$), and we spell out the statement at height $0$ for a reader not used to higher chromatic considerations:
\begin{thm}[Special case of \Cref{thm:Picshift} at $n=0$]
    Let $F$ be a separably closed field of characteristic $\neq p$. The inclusion $\mu_{p^\infty}(F)\subset F$ and the canonical map $\PPic(F)\to~GL_1(K(F))$ induce equivalences $$B\mu_{p^\infty}\to \Gpic(F)[p^\infty]\to \Gm(L_{K(1)}K(F))[p^\infty]\simeq \Gm(\KU_p)[p^\infty]$$
\end{thm}
That $\Gm(\KU_p)[p^\infty]$ is abstractly $B\mu_{p^\infty}$ is a direct consequence of \cite[Theorem H]{ChroNS}, and the fact that the map is the correct one can be seen by analyzing the proof thereof and thinking of the $p$-adic version of Snaith's theorem describing $\KU_p$ as $\Sph[B\mu_{p^\infty}]_p[\beta^{-1}]$.

The general version of \Cref{thm:Picshift} will be a simple consequence of the computations from \cite{ChroNS} and the ideas of \cite{Cyclored}. First, we need a bit of set-up. 
\begin{nota}
    Recall the Carmeli--Schlank--Yanovski height $n$ chromatic cyclotomic extensions \cite[Definition 4.7]{Cyclochro} denoted $\Sph_{T(n)}[\omega^{(n)}_{p^r}]$ and their colimit as $r$ increases, denoted $\Sph_{T(n)}[\omega^{(n)}_{p^\infty}]$. They are (compatibly) $(\Z/p^r)^\times$-Galois extensions of $\Sph_{T(n)}$ and are by design localizations of $\Sph_{T(n)}[B^n\Z/p^r]$ in a $(\Z/p^r)^\times$-equivariant way. 

    In particular, they come equipped with canonical and compatible $(\Z/p^r)^\times$-equivariant maps $\Sigma^n\Z/p^r\to GL_1(\Sph_{T(n)}[\omega^{(n)}_{p^r}])$. 

    At height $n=0$, we make a convention choice that they are $k(\zeta_{p^r})$, where $k$ is a prime field of appropriate characteristic $\neq p$ (the precise characteristic depending on that of the field $F$ we are trying to prove \Cref{thm:Picshift} for). 
\end{nota}
Recall that Lubin-Tate theories $E$ come with a determinant morphim $$\det: \Aut(E)\to \Z_p^\times$$ which admits splittings. The following lemma shows that this is compatible with the chromatic cyclotomic extensions. 
\begin{lm}\label{lm:detcomp}
    Let $E$ be a Lubin-Tate theory of height $n$, based on a (perfect) field (of characteristic $p$). There is a map of commutative ring spectra $$\Sph_{T(n)}[\omega^{(n)}_{p^\infty}]\to E$$

    Furthermore, any two such maps differ by an automorphism of $E$. In particular, any such map admits a canonical $\Z_p^\times$-refinement for some splitting $e:\Z_p^\times\to \Aut(E)$ of the determinant representation $\det:\Aut(E)\to\Z_p^\times$. 
\end{lm}
\begin{proof}
The first part follows from Hopkins--Lurie's version of the chromatic Fourier transform \cite[Corollary 5.3.26]{hopkinslurie} which implies in particular that $E[B^nC_{p^r}]\simeq E^{C_{p^r}}$ as commutative $E$-algebras, so that $E[\omega^{(n)}_{p^r}]\simeq E^{(C_{p^r})^\times}$ as commutative $E$-algebras, compatibly with $r$. In particular, $E[\omega^{(n)}_{p^\infty}]\simeq C(\Z_p^\times,E)$ where $C(X,E)$ is continuous maps from $X$ to $E$ for any profinite set $X$. 

Furthermore, this equivalence is $\Z_p^\times$-equivariant and $E$ has no nontrivial idempotents, so any map $C(\Z_p^\times,E)\to E$ is given by evaluation at a single $x\in\Z_p^\times$, and therefore any two such maps do differ by a translation in $\Z_p^\times$, hence, ultimately, an automorphism of $E$. 

The final statement follows formally by combining what was said so far with \cite[Theorem 5.8]{Cyclochro}. 
\end{proof}
We now state explicitly the following - the point of this observation is not the equivalence itself (which is again a direct consequence of \cite[Theorem H]{ChroNS}), rather the fact that it is induced by a specific map: 
\begin{obs}\label{obs:chroNSmap}
Let $E$ be a Lubin-Tate theory of height $n$ based on an algebraically closed field of characteristic $p$. 
    The composite map $\Sph_{T(n)}[B^n\Z/p^\infty]\to\Sph_{T(n)}[\omega^{(n)}_{p^\infty}]\to E$ corresponds by adjunction to a map $\Sigma^n \Z/p^\infty\to \Gm(E)$ and hence $\Sigma^{n+1}\Z/p^\infty\to \Gpic(E)$, and finally, again by adjunction, to a map $$\Sigma^{n+1}\Z/p^\infty\to\Gpic(E)[p^\infty]$$
    This last map is an equivalence.
\end{obs}
\begin{proof}
    This follows from the proof of \cite[Proposition 8.14]{ChroNS}, cf. also \cite[Remark 8.18]{ChroNS}

    We note that this also works when\footnote{Cf. also \cite[Propostion 3.8]{Shacharstrict} for a comparison with $\PPic[p^\infty]$ in this case. } $n=0$. 
    \end{proof}
We can now already prove \Cref{thm:Picshift}, which we recall for the convenience of the reader:
\begin{thm}\label{thm:Picshifttext}
Fix $n\geq 0$. 
    Let $E_n$ be any Lubin-Tate theory based on an algebraically closed field of characteristic $p$ if $n\geq 1$, or a separably closed field of characteristic $\neq p$ if $n=0$. For any map of commutative ring spectra $L_{T(n+1)}K(E_n)\to E_{n+1}$ to a height $(n+1)$ Lubin-Tate theory, also based on an algebraically closed field of characteristic $p$, the induced map $$\Gpic(E_n)[p^\infty]\to \Gm(L_{T(n+1)}K(E_n))[p^\infty]\to \Gm(E_{n+1})[p^\infty]$$  is an equivalence.
\end{thm}
\begin{proof}
    Since the result is clearly invariant under self-equivalences of $E_{n+1}$, \Cref{lm:detcomp} guarantees that we may assume without loss of generality that the following diagram commutes: 
    \[\begin{tikzcd}
	{} & {\Sph_{K(n+1)}[\omega^{(n+1)}_{p^{\infty}}]} \\
	{L_{K(n+1)}K(E_n^{h\Z_p^\times})[\omega^{(n+1)}_{p^\infty}]} \\
	{L_{K(n+1)}K(E_n)} && {E_{n+1}}
	\arrow[from=1-2, to=2-1]
	\arrow[from=1-2, to=3-3]
	\arrow["\simeq"', from=2-1, to=3-1]
	\arrow[from=3-1, to=3-3]
\end{tikzcd}\]
where the left vertical arrow is induced by cyclotomic redshift \cite[Theorem B]{Cyclored}. Thus, by \Cref{obs:chroNSmap}, we are reduced to the following lemma.
\end{proof}
\begin{lm}
    Let $R$ be a $T(n)$-local commutative ring spectrum. The following diagram of spectra commutes: 
    \[\begin{tikzcd}
	{\Sigma\Sigma^n\Z/p^\infty} & {\Sigma GL_1(R[\omega^{(n)}_{p^\infty}])} & {\PPic(R[\omega^{(n)}_{p^\infty}])} \\
	&& {GL_1(L_{T(n+1)}K(R[\omega^{(n)}_{p^\infty}]))} \\
	{\Sigma^{n+1}\Z/p^\infty} & {} & {GL_1(L_{T(n+1)}K(R)[\omega^{(n+1)}_{p^\infty}])}
	\arrow[from=1-1, to=1-2]
	\arrow["{=}"', from=1-1, to=3-1]
	\arrow[from=1-2, to=1-3]
	\arrow["K", from=1-3, to=2-3]
	\arrow["\simeq", from=2-3, to=3-3]
	\arrow[from=3-1, to=3-3]
\end{tikzcd}\]
and hence so does the following diagram of $\Z$-modules: 
\[\begin{tikzcd}
	{\Sigma\Sigma^n\Z/p^\infty} & {\Sigma \Gm(R[\omega^{(n)}_{p^\infty}])[p^\infty]} & {\Gpic(R[\omega^{(n)}_{p^\infty}])[p^\infty]} \\
	&& {\Gm(L_{T(n+1)}K(R[\omega^{(n)}_{p^\infty}]))[p^\infty]} \\
	{\Sigma^{n+1}\Z/p^\infty} & {} & {\Gm(L_{T(n+1)}K(R)[\omega^{(n+1)}_{p^\infty}])[p^\infty]}
	\arrow[from=1-1, to=1-2]
	\arrow["{=}"', from=1-1, to=3-1]
	\arrow[from=1-2, to=1-3]
	\arrow["K", from=1-3, to=2-3]
	\arrow["\simeq", from=2-3, to=3-3]
	\arrow[from=3-1, to=3-3]
\end{tikzcd}\]
\end{lm}
\begin{proof}
    The second diagram is a direct consequence of the first one by adjunction. The first one follows from the set-up of cyclotomic redshift, cf. \cite[Lemma 4.27]{Cyclored}. 
\end{proof}

\section{Galois theory and finitary behaviour of $\Gpic$}\label{section:Galois}
The goal of this section is to prove \Cref{thm:subfinintro}, that is, the subfinitaryness of the torsion part of $\Omega^{n-1}\Gpic$ on $T(n)$-local commutative ring spectra. We will do so by relating it via the chromatic Fourier transform \cite{Fourier} to a similar finitary-ness statement for Galois extensions. We begin with the necessary Galois-theoretic preliminaries and then move on to the actual proof. 
\subsection{Galois theory}
\newcommand{\C}{\mathcal{C}}
We begin with the Galois-theoretic preliminaries. Our goal will later be to apply them in the category of $T(n)$-local spectra. For clarity, we abstract the arguments away and work in the following setting: 
\begin{nota}\label{ass:standing}
    $\C\in\CAlg(\PrL_\st)$ is a stable presentably symmetric monoidal category which is compactly generated. We \emph{do not} assume that the unit is compact, but we do assume that there exists a dualizable object $V$ which carries the structure of an $\mathbb E_4$-algebra in $\C$, such that $V\otimes -$ is conservative and such that $V$ is compact in $\Mod_V(\C)$. 
\end{nota}
\begin{rmk}
    If $\C$ is as above, then so is $\Mod_R(\C)$ for every commutative algebra $R$ in $\C$. 
\end{rmk}
\begin{ex}\label{ex:T(n)standingass}
    In the case of $\C=\Sp_{T(n)}$, such a $V$ exists by the work of Burklund on Moore spectra \cite{RobertMoore}. 
\end{ex}
\begin{nota}\label{nota:idem}
    Let $C$ be an additively symmetric monoidal category and $R\in\Alg(C)$. We let $\Idem(R)$ denote the \emph{set} of idempotents in the ordinary ring $\pi_0\Map(\one_C,R)$. 

    When $R$ is homotopy commutative, this is equivalent to $\Map_{\Alg(\C)}(\one_C\times\one_C, R)$ and this is also true if one replaces $\Alg(\C)$ by $\Alg_{\mathbb E_n}(\C)$ for any $n\in\mathbb N_{\geq 2}\cup\{\infty\}$ (though then homotopy commutativity is automatic), cf. \cite{MeSep}. 
\end{nota}
\begin{lm}\label{lm:idem1}
    Let $\C$ be as in \Cref{ass:standing}, and let $R\in \CAlg(\C)$. The map $$\Idem(R)\to \Idem(R\otimes V)$$ is injective. 

    Furthermore, let $e\in \Idem(R\otimes V)$. It is in the image of $\Idem(R)$ if and only if $$(V\otimes R)_e \otimes_R (V\otimes R)_{1-e} = 0$$ 
    Here, for an $V\otimes R$-module $M$, we let $M_e := M[e^{-1}]$. 
\end{lm}
\begin{proof}
    Replacing $\C$ by $\Mod_R(\C)$ we may assume (for simplicity of notation) that $R=\one_\C$. 

    Let $e$ be an idempotent in $\one_\C$. The canonical map $M\to M[e^{-1}]$ is an equivalence if and only if it is one after tensoring with $V$. Thus $M$ is $e$-local if and only if $V\otimes M$ is $\overline{e}$-local, where $\overline{e}$ is the image of $e$ in $\Idem(V)$. Since one can recover $e$ from the category of $e$-local modules, the injectivity claim follows. 

    Let now $e\in\Idem(V)$ be such that $V[e^{-1}]\otimes V[(1-e)^{-1}] = 0$. What this assumption guarantees is that if $M$ is an $e$-local $V$-module, then so is $V\otimes M$ (viewed as a $V$-module via the left factor). Indeed, this statement reduces to $M=V[e^{-1}]$ and then this is clearly equivalent to the assumption. 
    
    Now define an object $M$ of $\C$ to be $\tilde e$-local if $V\otimes M$ is $e$-local, and similarly for $\widetilde{1-e}$. For $M$ $\tilde e$-local and $N$ $\widetilde{1-e}$, $\hom_\C(M,N)\otimes V\simeq \hom_{\Mod_V(\C)}(V\otimes M, V\otimes N) = 0$, where $\hom_\C$ denotes the internal hom. By conservativity of $V$, $\hom_\C(M,N) =0$, and the argument is symmetric. 

    Now suppose $N$ is right orthogonal to all $\tilde e$-local objects, and let $M$ be an $e$-local $V$-module. By our observation earlier, $V\otimes M$ is also $e$-local, so that $M$ is $\tilde e$-local as an object of $\C$, and hence $\hom_\C(M,N)= 0$ and hence $\hom_{\Mod_V(\V)}(M,V\otimes N)=0$, so that ultimately $N$ is $\widetilde{(1-e)}$-local. It follows that $\C$ splits as a product of $\tilde{e}$-local and $\widetilde{1-e}$-local objects, and so we get a corresponding idempotent in $\one_\C$. From the definitions and this last point, it is now clear that this idempotent maps to $e\in V$, as was to be shown. 
\end{proof}

 We can now prove the following key lemma, which is the base for an induction up the layer of Galois extensions. 
\begin{lm}\label{lm:idmcpct}
Let $\C$ be as in \Cref{ass:standing}. In this case, for any finite set $X$, the commutative algebra $\prod_X\one_\C$ is compact. 
\end{lm}
\begin{rmk}
The analogous statement for $\CAlg(\Sp)$ is immediate; the difference is of course that here, the functor $\pi_0\Map_\C(\one_\C,-)$ itself does not need to commute with filtered colimits. 
\end{rmk}
\begin{proof}
We begin with the empty set. In this case, $\Map_{\CAlg(\C)}(0,R)$ is either empty (if $R$ is nonzero) or contractible (if $R$ is $0$). Thus we must show that if $\colim_I R_i=0$, then some $R_i=0$. But this is easily proved by tensoring with $V$ and observing that the result is obvious in categories in which the unit is compact (note that this does not use any commutativity). 

We now move on to non-empty products. By a straightforward argument, we reduce to proving that $\one_\C\times\one_\C$ is compact, i.e. to the statement that $\Idem(-)$ preserves filtered colimits. 

So let $R_\bullet: I\to \CAlg(\C)$ be a filtered diagram with colimit $R$. Since $V$ is compact in $\Mod_R(\C)$, we have a colimit diagram $\colim_I \Idem(V\otimes R_i)\cong \Idem(V\otimes R)$. By \Cref{lm:idem1}, it therefore follows that $\colim_I \Idem(R_i)\to \Idem(R)$ is injective, and we are left with checking surjectivity. So let $e\in\Idem(R)$, and consider its image $\overline{e}$ in $\Idem(V\otimes R)$. 

By the first step, $\overline{e}$ lifts to some $\tilde e_i\in \Idem(V\otimes R_i)$. Consider now the $I_{i/}$-indexed system of algebras $(V\otimes R_j)_{\overline{e}}\otimes_{R_j} (V\otimes R_j)_{1-\overline{e}}$. Its colimit vanishes, hence it must vanish at some finite stage by the case of $X=\emptyset$ treated at the beginning of the proof. Therefore, for $j$ large enough $\tilde e_j$ is in the image of $\Idem(R_j)$ by \Cref{lm:idem1}, which proves surjectivity, as required. 
\end{proof}
Now we move on to more serious Galois theory. For this we briefly recall some notation and conventions. 
\begin{defn}
    Let $C\in\CAlg(\PrL)$ and $R\in\CAlg(C)$, as well as $G$ a group in spaces. A $G$-Galois extension of $R$ is an object $S\in\CAlg(C)^{BG}$ with a map $R^{\mathrm{triv}}\to S$ in $\CAlg(C)^{BG}$ satisfying the following two conditions:
    \begin{itemize}
        \item The induced map $R\to S^{hG}$ is an equivalence;
        \item The map $S\otimes_R S\to S^G$, mate of the action map $S\otimes_R S\otimes_R R[G]\to S\otimes_R S\to S$, is an equivalence. 
    \end{itemize}
    It is called faithful if $S\otimes_R -$ is conservative on $\Mod_R(C)$. 
\end{defn}
\begin{nota}
    We let $\Gal_G(R)$ denote the full subgroupoid\footnote{Most of our results also work for the full subcategory, but ultimately we plan to restrict to the full subgroupoid, so we define it this way for simplicity of notation.} of $(\CAlg(C)^{BG})_{R^{\mathrm{triv}}/}$ spanned by \emph{faithful} $G$-Galois extensions. 

    More generally, for $D\in\CAlg(\PrL)$, we let $\Gal_G(D)$ denote the full subgroupoid of $\CAlg(D)^{BG}$ spanned by Galois extensions of $\one_D$. We recover the previous example by picking $D=\Mod_R(C)$. 
\end{nota}
\begin{warn}\label{warn:notationGal}
    This notation conflicts with the notation from \cite{Fourier}, which is a source we will use later, so let us spell out the connection to clear out the confusion. The authors there consider (rightly so) that since the notion of Galois extension really only depends on the space $BG$ and not the group $G$, and since it could be fruitful to have a many-objects version of Galois extensions, one should define $A$-Galois extensions for a space $A$ (to be thought of as $BG$), and so what \emph{we} denote $\Gal_G(R)$ would be denoted $\Gal_{BG}(R)$ in \cite{Fourier}. For us, however, the group will come up more often, and so we stick to our notational choice.  
\end{warn}
One key fact we will use is Galois descent: 
\begin{lm}
    Let $D\in\CAlg(\PrL)$ and $R\in\CAlg(D)^{BG}$ a faithful $G$-Galois extension. Suppose $\one_D[G]$ is dualizable in $D$\footnote{This is the analogue of the ``stably dualizable'' assumption in \cite{rognesgalois}.} and $G$ is $D$-affine in the sense of \cite[Definition 2.15]{Fourier}.

    In this case, the canonical map $D\to \Mod_S(D)^{hG}$ is an equivalence of symmetric monoidal categories. 
\end{lm}
\begin{proof}[Sketch of proof]
    From the assumption that $\one_D[G]$ is dualizable, one deduces that $S$ is itself dualizable, cf. \cite[Proposition 6.2.1]{rognesgalois} whose proof is valid in that generality. 

Since $S$ is conservative, one can check that the unit map $M\to (S\otimes M)^{hG}$ is an equivalence after tensoring with $S$, and since $S$ is dualizable it becomes the same unit map but for $S\otimes M$, i.e. the map $S\otimes M\to (S\otimes S\otimes M)^{hG}\simeq (S^G\otimes M)^{hG}\simeq ((S\otimes M)^G)^{hG}\simeq S\otimes M$, where the last step again used dualizability of $\one_D[G]$ (one needs to check that the map is the correct one). 

Now we simply need to check that the right adjoint is conservative. The commutative algebra map $S\otimes S\to S^G$ is $G$-equivariant with respect the this action and the coinduced action on $S^G$, and is an equivalence, so this is equivalently $\Mod_{S^G}(D)^{hG}$. Since $G$ is $D$-affine, we can rewrite this as $(\Mod_S(D)^{G})^{hG}$ where the action on the inner term is clearly coinduced, so that we get simply $\Mod_S(D)$. In particular, the right adjoint is conservative after tensoring with $S$, and since tensoring with $S$ is conservative, the right adjoint is conservative altogether, as was to be shown. 
\end{proof}
\begin{lm}\label{lm:compactIdemcompactGalois}
Let $C\in\CAlg(\PrL)$ be an additively presentably symmetric monoidal category, and assume that $R\mapsto\Idem(R)$ is filtered-colimit-preserving. Then for any finite group $H$, and any faithful $H$-Galois extension $R\to S$ in $C$, $S$ is compact over $R$. 
\end{lm}
\begin{proof}
We have an equivalence of sets $\Map_{\CAlg_R(C)}(S,-)\simeq \Map_{\CAlg_R(C)}(S, S\otimes -)^{hH}$ and $H$-fixed points commute with filtered colimits in sets, so it suffices to prove that $\Map_{\CAlg_R(C)}(S, S\otimes_R-)$ also does. But this is equivalent to $\Map_{\CAlg_S(C)}(S\otimes_R S, S\otimes_R-)$ and so ultimately it suffices to prove that $S\otimes_R S\simeq \prod_H S$ is a compact commutative $S$-algebra. This simply follows from the idempotent hypothesis. 
\end{proof}
The previous lemma is an instance of the following claim: if standard/trivial Galois extensions are compact in $C$, then all Galois extensions are compact. It used the fact that $H$-fixed points commutes with filtered colimits on \emph{sets}, so to generalize this we will first need to prove a truncatedness property. 
\begin{lm}\label{lm:Galoistrunc}
Let $C\in\CAlg(\PrL_\st)$ and let $G$ be a truncated group in spaces such that $G$ and all its iterated loop spaces $\Omega^k(G,1)$ are $C$-affine and $\one_C[\Omega^k(G,1)]$ are all dualizable (including $k=0$, where we interpret this as just $G$).

In this case, any faithful $G$-Galois extension in $C$ is cotruncated, that is, if $R\to S$ is faithful $G$-Galois, then $\Map_{\CAlg_R(C)}(S,-)$ has values in truncated spaces.     
\end{lm}
\begin{proof}
    We proceed by induction on the truncated-ness of $G$. 

    Let $R\to S$ be a $G$-Galois extension and $T$ a commutative $R$-algebra. By Galois descent, we have $\Map_{\CAlg_R}(S,T)\simeq \Map_{\CAlg_R}(S,S\otimes_R T)^{hG}$ so it suffices to prove the claim for $T$ an $S$-algebra. In that case $$\Map_{\CAlg_R}(S,T)\simeq \Map_{\CAlg_S}(S\otimes_R S,T)\simeq \Map_{\CAlg_S}(F(G,S),T)\simeq \Map_{\CAlg}(F(G,\one_C),T)$$ so it suffices to prove the claim for the trivial Galois extension $F(G,\one_C)$. But now the map $F(G,\one_C)\otimes F(G,\one_C)\to F(G,\one_C)$ is equivalent, via shearing, to $F(G,\one_C)\otimes (F(G,\one_C)\to~\one_C)$ so it suffices to prove that the map $F(G,\one_C)\to \one$ is cotruncated, and thus, using the fact that products are cotruncated, we reduce to the same statement about $F(G_1, \one_C)$, where $G_1$ is the component of $1$ in $G$. The map $F(G,\one_C)\to \one_C$ is now a faithful (using affineness) $\Omega(G,1)$-Galois extension and therefore the claim follows by our induction hypothesis.
\end{proof}
\begin{cor}
    Let $C\in\CAlg(\PrL_\st)$ and $G$ a truncated, finite type group in spaces such that $G$-Galois extensions in $C$ are cotruncated. Suppose $\one_C[G]$ is dualizable and $G$ is $C$-affine. 
    
    Suppose further that $F(G,\one_C)$ is a compact commutative algebra in $C$. In this case, for any $G$-Galois extension $R\to S$ in $C$, $S$ is compact in $\CAlg_R(C)$.
\end{cor}
\begin{proof}
    This is the same argument as in \Cref{lm:compactIdemcompactGalois}, using the truncatedness. 
\end{proof}
\begin{lm}\label{lm:calgcolim}
    Let $C_\bullet: I\to\CAlg(\PrL)$ be a filtered diagram with colimit $C_\infty$. The induced map $\colim_I \CAlg(C_i)\to \CAlg(C_\infty)$ is an equivalence in $\PrL$. 
\end{lm}
\begin{proof}
    We consider instead the diagram of right adjoints and the induced map $$\CAlg(C_\infty)\to \lim_{I\op}\CAlg(C_i).$$ 

    It suffices to prove that $C_\infty^{\otimes}\to \lim_{I\op}C_i^\otimes$ is a limit diagram of $\infty$-operads. Unwinding what this means, we reduce to proving the following claim, where $f_i^R$ denotes the right adjoint of the canonical map $f_i:C_i\to C_\infty$: for all $x_1,...,x_n, y\in C_\infty$, the map $$\Map_{C_\infty}(\otimes_k x_k, y)\to \lim_{I\op}\Map_{C_i}(\otimes_k f_i^R(x_k),f_i^R(y)) $$
    is an equivalence. This amounts to the statement that $\colim_i f_i(\otimes_k f_i^R(x_k))\simeq~\otimes_kx_k$. Since $f_i$ is \emph{strong} symmetric monoidal this amounts to the statement that $\colim_i \otimes_k f_if_i^R (x_k)\simeq~\otimes_k x_k$. Since $I$ is filtered and $\otimes$ commutes with colimits in each variable, we reduce to the statement that for each $x$, $\colim_i f_if_i^R(x)\to x$ is an equivalence, which follows itself from $C_\infty\simeq \lim_{I\op}C_i$ along the right adjoints. 
\end{proof}
\begin{nota}
    Let $\CAlg(\PrL)_{(\omega)}$ denote the pullback $\CAlg(\PrL)\times_{\PrL}\PrL_\omega$.
    Note that this category admits filtered (even sifted) colimits, that are computed underlying. 
\end{nota}
\begin{lm}\label{lm:calgcpt}
      The functor $$\CAlg(\PrL)_{(\omega)}\to\Cat, \mathcal D\mapsto\CAlg(\mathcal D)^\omega$$ is well defined and commutes with filtered colimits. 
\end{lm}
\begin{proof}
That it is well-defined follows from the fact that for $\mathcal D\in\CAlg(\PrL)_{(\omega)}$, $\CAlg(\mathcal D)$ is compactly generated by free commutative algebras on the compact generators of $\mathcal D$, which are certainly preserved by functors in $\CAlg(\PrL)_{(\omega)}$.

    The forgetful functor $\Cat^{\mathrm{rex, idem}}\to \Cat$ from the category of finitely cocomplete, idempotent-complete categories preserves filtered colimits, so it suffices to prove it with values in the former category. But now $\Ind(-)$ induces an equivalence between that former category and $\PrL_\omega$, so it suffices to prove it there. 

    Now $\PrL_\omega\to \PrL$ also preserves colimits, so ultimately the claim simply follows from the fact that $\mathcal D\mapsto \CAlg(\mathcal D)$ preserves filtered colimits, which is \Cref{lm:calgcolim}. 
\end{proof}
To state the following result, we recall the definition from the introduction: 
\begin{defn}
Let $F:C\to D$ be a functor between categories with filtered colimits. $F$ is called \emph{subfinitary} if for every filtered diagram in $C$, the induced map $\colim_I F(c_i)\to~F(\colim_I c_i)$ is a monomorphism in $D$. 
\end{defn}
\begin{rmk}
    Recall that a morphism $x\to y$ in a category $D$ is a monomorphism if the square 
    \[\begin{tikzcd}
	x & x \\
	x & y
	\arrow[from=1-1, to=1-2]
	\arrow[from=1-1, to=2-1]
	\arrow[from=1-2, to=2-2]
	\arrow[from=2-1, to=2-2]
\end{tikzcd}\] 
is a pullback square. In the category of spaces, this is equivalent to being an inclusion of connected components. 
\end{rmk}
\begin{ex}\label{ex:subfunctor}
If $G$ is a finitary functor, i.e. $G$ commutes with filtered colimits, and if filtered colimits in $D$ are left exact, then any subfunctor of $G$ is subfinitary. 
\end{ex}
\begin{prop}\label{prop:bootstrap}
Let $\C$ be as in \Cref{ass:standing}, and assume furthermore that $\C$ is compactly generated. Let $G$ be a truncated group in spaces such that:
\begin{itemize}
    \item $BG$ all its iterated loop spaces are $\C$-adjointable in the sense of \cite{CCRY}, that is, $\lim_{BG}$ and $\lim_G$ preserve colimits and tensors with objects of $\C$; 
    \item $BG$ is of finite type, that is, $\tau_{\leq k}BG$ is compact in $\Ss_{\leq k}$ for all $k$;
    \item All the iterated loop spaces of $G$ are $\C$-affine;
    \item $F(G,\one_\C)$ is a compact commutative algebra in $\C$
\end{itemize}

In this case, the assignment $\mathcal D\mapsto \Gal_G(\mathcal D)$ is subfinitary, as a functor $(\CAlg(\PrL)_{(\omega)})_{\mathcal C/}\to~\Ss$.  

If $G$ is a \emph{discrete} group, necessarily finite given the other hypotheses, then the above functor is in fact finitary. 
\end{prop}
\begin{proof}
\newcommand{\ho}{\mathrm{ho}}
We first note that all the properties of $G$ relative to $\C$ are inherited by $\mathcal D$'s equipped with a map $\C\to \mathcal D$ in $\CAlg(\PrL)$. 

We then note that the first bullet point also guarantees that $\one_\C[G]$ is dualizable, cf. \cite[Proposition 4.33]{CCRY}, and similarly for all the iterated loop spaces of $G$.  

Thus by \Cref{lm:Galoistrunc} and the third bullet point, all $G$-Galois extensions are cotruncated. 

It follows then by combining Galois descent, compactness of $\tau_{\leq k} BG$ and compactness of $F(G,\one_\C)$ in commutative algebras that all $G$-Galois extensions are compact commutative algebras. 

As a final preliminary observation, we note that \Cref{lm:calgcpt} states exactly that $\CAlg(-)^\omega$ is a finitary functor. 

Notice that for any $d$, the constant colimit $\colim_{S^{d+1}}$ in commutative algebras preserves compact commutative algebras. Hence, the property of being $d$-cotruncated, which is equivalent to the fold map $\colim_{S^{d+1}} A\to A$ being an equivalence, is preserved by the transition functors, so that we get a subfunctor of $d$-cotruncated commutative algebras. This is a finitary subfunctor of $\CAlg(-)^\omega$, and it has values in $(d+1)$-categories, so that applying $(-)^{BG}$ to it preserves finitariness. 

Finally, $\Gal_G(-)$ is a subfunctor of this functor of $d$-cotruncated commutative algebras with a $G$-action, and hence by \Cref{ex:subfunctor}, is a subfinitary functor. 

We are left with the statement that for discrete, finite $G$, this functor is actually finitary. Because the unit is only compact after tensoring with $V$, we will need to fight a bit. We fix colimit diagram $\mathcal D_\infty\simeq \colim_i \mathcal D_i$, and $\one_{\mathcal D_\infty}\to S$ a $G$-Galois extension in $\mathcal D_\infty$. By subfinitariness, we only need to prove that it lifts to a $G$-Galois extension $S_i$ of $\one_{\mathcal D_i}$ for some $i$.

First, we note that what we said so far had little to do with commutative algebras and so could be done also with $\mathbb E_3$-algebras, say. To make everything truly work, we simply note that for $G$ a discrete group and $S$ a $G$-Galois extension of $\one_{\mathcal D_\infty}$, $S$ is \emph{also} compact and $0$-cotruncated as an $\mathbb E_3$-algebra. This follows from \cite[Corollary 1.3.39]{MeSep}.

Thus, by finitariness of $C\mapsto \Alg_{\mathbb E_3}(C)^\omega$, we already know that there is a compact, $0$-cotruncated $\mathbb E_3$-algebra with $G$-action $S_i$ which basechanges to $S$. We are going to try to prove that $S_i$ is underlying dualizable. 

For this we switch gears a little bit and tensor with the $V$ that is guaranteed by \Cref{ass:standing}. We note that with a compact unit, we can consider the functor  $C\mapsto \CAlg(\ho(C^\omega))$, which is easily seen to be finitary, because of the $\ho$. Thus,  $S\otimes V$ admits a lift $T_i$ to some $\CAlg(\ho(\Mod_V(\mathcal D_i)^\omega))$. Since $S$ and hence $S\otimes V$ is homotopy separable in $\Mod_V(\mathcal D_\infty)$, in the sense of \cite[Variant 1.1.3]{MeSep}, $T_i$ can be chosen to be homotopy separable for $i$ large enough, and thus it has, by \cite[Theorem 1.3.19]{MeSep}, a unique lift to an $\mathbb E_3$-algebra $T_i\in\Alg_{\mathbb E_3}(\Mod_V(\mathcal D_i))$, which is underlying compact. Furthermore, the map $T_i\otimes T_i\to \prod_G T_i$ is now a map between compact objects in $\Mod_V(\mathcal D_i)$, which becomes an equivalence in the colimit and hence is an equivalence at some finite stage. 

But now after tensoring with $V$, $(-)^{hG}$ preserves compacts and therefore also the map $\one_{\mathcal D_i}\to T_i^{hG}$ is a map between compacts which becomes an equivalence in the colimit and hence becomes an equivalence at a finite stage. In other words, $T_i$ is a(n $\mathbb E_3$-variant of a) $G$-Galois extension of $V\otimes\one_{\mathcal D_i}$. 

In particular, by \cite[Theorem 1.3.39]{MeSep}, separability and underlying compactness guarantee that $T_i $ is actually compact as an $\mathbb E_3$-algebra. So now, by finitariness of the functor $C\mapsto \Alg_{\mathbb E_3}(C)^\omega$, it follows that $S_i\otimes V\simeq T_i$, $G$-equivariantly for $i$ large enough, and therefore that $S_i\otimes S_i\to \prod_G S_i$ is an equivalence for $i$ large enough, since it is an equivalence after tensoring with $V$, and the same is true for $\one_{\mathcal D_i}\to S_i^{hG}$. 

In particular, the first point guarantees that $S_i$ is separable, cf. \cite[Section 1.5.1]{MeSep}, and so by \cite[Theorem 1.3.19]{MeSep} lifts uniquely to a commutative algebra, and its $G$-action too. Thus it is a faithful $G$-Galois extension which maps to $S$, as was to be constructed.
\end{proof}
\begin{ex}\label{ex:T(n)cpct}
    For every $n$, $\Sp_{T(n)}$ is compactly generated, though its unit is not compact. It satisfies the hypotheses of the previous proposition for every $n$-truncated $p$-finite group $G$. 
\end{ex}
In fact, the above example can be generalized: 
\begin{lm}\label{lm:etcpct}
    Let $X$ be a qcqs scheme. The category $\Sh(X_\et,\Sp_{T(n)})$ of étale sheaves on $X$ with values in $T(n)$-local spectra is compactly generated.

    Furthermore, for any étale map $Y\to X$, the pullback functor $\Sh(X_\et,\Sp_{T(n)})\to \Sh(Y_\et,\Sp_{T(n)})$ preserves compacts, or equivalently has a colimit-preserving right adjoint. 
\end{lm}
\begin{proof}
    The category of étale presheaves $\Psh(X_\et,\Sp_{T(n)})\simeq \Psh(X_\et)\otimes\Sp_{T(n)}$ is compactly generated by the previous example. 

    It therefore suffices to prove that the inclusion functor $\Sh(X_\et,\Sp_{T(n)})\to \Psh(X_\et,\Sp_{T(n)})$ preserves colimits, i.e. that sheaves are closed under colimits. By \cite[Corollary B.7.6.2]{SAG}, $\mathcal F$ is an étale sheaf if and only if it is a Nisnevich sheaf and satisfies Galois descent. The first condition is, by \cite[Theorem 3.7.5.1]{SAG}, a condition that certain squares be sent to pullbacks, and since $\Sp_{T(n)}$ is stable, this is closed under colimits. The second condition is the a priori complicated one (which fails for $\Sp$ in place of $\Sp_{T(n)})$ as it involves fixed points for a finite group. However, $\Sp_{T(n)}$ is $1$-semiadditive by Kuhn's Tate vanishing theorem \cite[Theorem 1.5]{kuhn}, and therefore $(-)^{hG}: \Sp_{T(n)}^{BG}\to\Sp_{T(n)}$ preserves all colimits. 

    For the ``furthermore'' part, consider the following commutative diagram: 
    \[\begin{tikzcd}
	{\Psh(X_\et,\Sp_{T(n)})} & {\Psh(Y_\et,\Sp_{T(n)})} \\
	{\Sh(X_\et,\Sp_{T(n)})} & {\Sh(Y_\et,\Sp_{T(n)})}
	\arrow[from=1-1, to=1-2]
	\arrow[from=1-1, to=2-1]
	\arrow[from=1-2, to=2-2]
	\arrow[from=2-1, to=2-2]
\end{tikzcd}\]
Being a left Kan extension, the top functor has a colimit-preserving right adjoint. Both vertical arrows have colimit-preserving, fully faithful right adjoints. It follows that the bottom map also has a colimit-preserving right adjoint (this is elementary, but see also \cite[Lemma 1.30]{Dbl}). 
\end{proof}

\subsection{The chromatic Fourier transform and $\Gpic$}
In this section, we begin with brief recollections on Barthel--Carmeli--Schlank--Yanovski's chromatic Fourier transform from \cite{Fourier}, and we follow them with our main application, namely the key finitary-ness property \Cref{thm:subfinintro}.  

The following was inspired by discussions with Shai Keidar, who proved the same result (though as far as I understand, earlier): 
\begin{prop}\label{prop:Kummer1up}
   Let $H$ be an ordinary $\Z/p^r$-module. For every $0\leq k\leq n$, there is an equivalence, natural in the $T(n)$-local commutative $\Sph_{T(n)}[\omega^{(n)}_{p^r}]$-algebra $R$, of the form $$\Gal_{B^kH^*}(R)\simeq \Map(\Sigma^{n-k}H,\PPic(R))$$
   More generally, one can replace $R$ by $C\in \CAlg(\PrL)_{\Mod_{\Sph_{T(n)}[\omega^{(n)}_{p^r}]}(\Sp_{T(n)})/}$ and get a similar natural equivalence $$\Gal_{B^kH^*}(C)\simeq \Map(\Sigma^{n-k}H,\PPic(C))$$
\end{prop}
\begin{rmk}
    The proof of this proposition makes heavy use of \cite{Fourier} and its terminology. For this reason, it is difficult to recall all the needed terminology without essentially copying large swaths of \textit{loc. cit.}. We shall not need the specifics of the proof later on, only the statement, so the reader unfamiliar with \cite{Fourier} is invited to take the above proposition as a blackbox on first reading.  
\end{rmk}
\begin{proof}
Let $C\in\CAlg(\PrL)$ be equipped with a map from $\Mod_{\Sph_{T(n)}[\omega^{(n)}_{p^r}]}(\Sp_{T(n)})$. By \cite[Proposition 7.27,Corollary 6.5, Corollary 6.7, Proposition 4.5, Corollary 5.16]{Fourier}, 
 we find that: \begin{center}
      $\Sp_{T(n)}[\omega^{(n)}_{p^r}]$ is $(\Z/p^r,n)$-oriented, and therefore $\Mod_C(\PrL)$ is itself $(\Z/p^r,n+1)$-oriented.
 \end{center}

We conclude from this the following two facts: 
\begin{itemize}
    \item By \cite[Proposition 4.30]{Fourier}, $B^kH$ is $\Mod_C(\PrL)$-affine for all $0\leq k\leq n+1$; 
    \item The Fourier transform gives a canonical symmetric monoidal, $C$-linear equivalence $C[\Sigma^{n-k}H]\simeq C^{B^{k+1}H^*}$ - here we use that $I^{(n+1)}_p(\Sigma^{n-k}H)\simeq B^{k+1}H^*$. 
\end{itemize}
The first bullet point together with \cite[Proposition 2.30]{Fourier} guarantees that $C^{B^{k+1}H^*}$ corepresents, in $\CAlg_C$, $B^kH^*$-Galois extensions\footnote{Note that unlike in \cite{Fourier}, we call $G$-Galois extensions the notion where the Galois \emph{group} is $G$, as opposed to the Galois space. This explains the shift by $1$, cf. \Cref{warn:notationGal}.}.

On the other hand, $C[\Sigma^{n-k}H]$ clearly corepresents $\Map(\Sigma^{n-k}H,\PPic(-))$.

To conclude, we now simply have to note that if $D$ is a commutative $C$-algebra and $F:C^{B^{k+1}H^*}\to D$ a commutative $C$-algebra map, the induced Galois extension of $D$, as a commutative algebra in $\Mod_C(\PrL)$, is of the form $\Mod_R(D)$ for some Galois extension \emph{in} $D$. For this, we let $D\to D'$ be a $B^kH^*$-Galois extension, and note that the precomposition with the right adjoint $D'\to D\to D'$ is equivalent to the composite $D'\to \Fun(B^kH^*, D')\xrightarrow{\lim}D'$ where the first functor is induced by the $B^kH^*$-action - this follows since $D\simeq (D')^{hB^kH^*}$. In particular, since $B^kH^*$ is $D'$-affine, this composite is conservative and hence so is the right adjoint $D'\to D$. 

Since it is a $D$-commutative algebra basechanged from $C^{B^{k+1}H^*}\to C$, which clearly has the projection formula, it also has the projection formula and hence conservativity guarantees that it is affine, hence of the form $\Mod_R(D)$ for some $B^kH^*$-Galois extension $R$ of $\one_D$. \end{proof}

We now use the chromatic Fourier transform to relate Galois extensions and strict Picard elements to obtain a proof of \Cref{thm:subfinintro}:
\begin{thm}\label{thm:finitary}
    For any finite $p$-power torsion abelian group $H$, the functor $R\mapsto \Map_\Z(\Sigma^{n-1}H,\Gpic(R))$, as a functor of $T(n)$-local commutative algebras, is subfinitary.

    Here, $\Gpic(R):= \Map_{\Sp_{\geq 0}}(\Z,\PPic(\Mod_R(\Sp_{T(n)})))$. 
\end{thm}
\begin{rmk}
As mentioned in the introduction, we believe that much more is true, namely that one should be able to remove the $\Sigma^{n-1}$ appearing in the statement, and one should be able to prove that the resulting functor is finitary as opposed to subfinitary. 
    However, the functor $R\mapsto \PPic(L_{K(1)}R)$ does \emph{not} preserve filtered colimits. To wit, $\Sigma^2 R$ is different from $R$ for all finite Galois extensions of $\Sph_{K(1)}$, but it is equivalent to $R$ for $R=\KU_p$. Since $\Sigma R$ is $2$-torsion, this is, at the prime $2$, also an example to show that $\PPic[2]$ does not preserve filtered colimits. 
 
    However, $\Sigma^2 R$ is not strict. We do not know whether $\Gpic(L_{K(1)}(-))$ preserve filtered colimits. 
\end{rmk}

\begin{rmk}
   The classical Kummer theory from \cite[Section 3.3]{Cyclochro} shows that we can also interpret prime-to-$p$ torsion in $\Gpic$ in terms of Galois extensions but here always with discrete groups, so that we also find finitariness, but now with the same proof for the whole $\Gpic$ (no loops, and no subfinitariness).
\end{rmk}

We begin with a lemma:
\begin{lm}
    Let $H$ be a finite $p$-power torsion abelian group and $R$ a $T(n)$-local commutative ring spectrum. The connective mapping spectrum $\Map(H,\PPic(R))$ is $n+1$-truncated. 
\end{lm}
\begin{proof}
    The proof is the same as in \cite[Proposition 8.14]{ChroNS}. Namely, consider $m\geq n+2$ and note that $$\pi_m(\Map(H,\PPic_{T(n)}(R))\simeq \pi_0\Map(\Sigma^m H,\PPic_{T(n)}(R)) \simeq \pi_0 \Map_{\CAlg(\Sp_{T(n)}}(\Sph_{T(n)}[B^{m-1}M], R) $$ and 
$\Sph_{T(n)}[B^{m-1}R]\simeq \Sph_{T(n)}$ by \cite[Theorem 5.3.5]{AmbiChro} as soon as $m-1\geq n+1$, i.e. $m\geq n+2$. 
\end{proof}

\begin{proof}[Proof of \Cref{thm:finitary}]
We fix a finite $p$-power torsion abelian group $H$, say with exponent $p^r$. 

 We start by observing that by descent, we may assume without loss of generality that $R$ admits primitive $p^r$th roots of unity: indeed, since $\Map(\Sigma^{n-1}H,\PPic(R[\omega^{(n)}_{p^r}]))$ is truncated, if it is subfinitary, so are its $(\Z/p^r)^\times$ fixed points, and we have, by Galois descent, using \cite[Proposition 5.2]{Cyclochro}, $$\Map(\Sigma^{n-1}H,\PPic(R))\simeq \Map(\Sigma^{n-1}H,\PPic(R[\omega^{(n)}_{p^r}]))^{h(\Z/p^r)^\times}.$$

 In other words we may work in commutative $\Sph_{T(n)}[\omega^{(n)}_{p^r}]$-algebras instead. 

 Here, we are going to prove the result by using the chromatic Fourier transform. Indeed, by \Cref{prop:Kummer1up}, we have a natural equivalence $$\Map(\Sigma^{n-1}H,\PPic(R))\simeq \Gal_{B H^*}(R)$$ so we need to prove that this space of Galois extensions is subfinitary. By \Cref{prop:bootstrap} it suffices to show that $\one^{B H^*}$ is compact: the other assumptions of that proposition are satisfied by \Cref{ex:T(n)standingass}, \Cref{ex:T(n)cpct}, $T(n)$-local ambidexterity \cite[Theorem A]{AmbiChro}, finiteness of $H$ and \cite[Theorem 7.29]{Fourier}. 

 By $T(n)$-local affineness of $BH^*$, \cite[Proposition 2.30]{Fourier}\footnote{\Cref{warn:notationGal} is relevant here again.} guarantees that maps out of $\one^{BH^*}$ corepresent $H^*$-Galois extensions. Now \Cref{prop:bootstrap}, in the case of a discrete group, guarantees that $\Gal_{H^*}(-)$ is finitary if $\one^{H^*}$ is compact, which we proved in \Cref{lm:idmcpct}. 
\end{proof}
In fact, we can go a bit further, and this will be convenient when we study categories of sheaves. Indeed, the same proof together with the general version of \Cref{prop:bootstrap} and of \Cref{prop:Kummer1up} shows: 
\begin{cor}\label{cor:catfinitary}
    Let $H$ be a finite $p$-power torsion abelian group. The functor $$\mathcal C\mapsto \Map_\Z(\Sigma^{n-1}H,\Gpic(\mathcal C))$$ is subfinitary, as a functor on $(\CAlg(\PrL)_{(\omega)})_{\Sp_{T(n)}/}$. 
\end{cor}
\begin{proof}
    Exactly the same proof - we first show that this space is truncated, and this allows us to reduce to the case of commutative $\Mod_{\Sph_{T(n)}[\omega^{(n)}_{p^r}]}(\Sp_{T(n)})$-algebras, where we use \Cref{prop:Kummer1up} to reduce to $\mathcal C\mapsto \Gal_{H^*}(\mathcal C)$ being finitary, which is covered by \Cref{prop:bootstrap}. 
\end{proof}
\section{$K(1)$-local $K$-theory of Azumaya algebras}\label{section:main}
\subsection{The Künneth formula}
In this section, we discuss the Künneth formula from \Cref{thm:Kintro}. 

The basic instance of Künneth formula which motivates our whole approach is the following: 
\begin{lm}\label{lm:GaloisKünneth}
    Let $X$ be a qcqs scheme with $p$ invertible and $Y\to X$ a $G$-Galois cover where $G$ is a $p$-group. For any $X$-linear category $C$, The canonical map $K(Y)\otimes_{K(X)}K(C)\to K(\Perf(Y)\otimes_{\Perf(X)}C)$ is a $K(1)$-local equivalence.
\end{lm}
\begin{proof}
    We note that by \cite[Theorem B]{CMNN} and the fact that tensor products commute with fixed points in $\Sp_{K(1)}$ (which follows from $K(1)$-local Tate vanishing, cf. \cite{GS}, the canonical map becomes an equivalence upon taking $G$-fixed points, i.e. limit over $BG$. By \cite[Theorem 7.29]{Fourier}, $BG$ is $K(1)$-locally affine and so taking fixed points is conservative, thus the claim follows. 
\end{proof}

By combining it with the Künneth formula from \cite{BCM}, we obtain \Cref{thm:Kintro}:
\begin{cor}\label{cor:AzKünneth}
     Let $X$ be a qcqs scheme with $p$ invertible and $A$ an Azumaya algebra over $X$. Assume the pullback of $A$ to $X\times_{\Spec(\Z[\frac{1}{p}])}\Spec(\Z[\frac{1}{p},\zeta_{p^\infty}])$ of the latter scheme is split by a $G$-Galois cover for some $p$-group $G$. In this case, for every $X$-linear category $C$, the map $K(A)\otimes_{K(X)}K(C)\to K(\Perf(A)\otimes_{\Perf(X)}C)$ is a $K(1)$-local equivalence. 
\end{cor}
In this situation, we say $A$ \emph{satisfies (the) $K(1)$-local Künneth (formula)}. Note that this only depends on $\Perf(A)$, i.e. on $A$ up to Morita equivalence, thus we may say that a Brauer class satisfies $K(1)$-local Künneth. 
\begin{proof}
Since $\KU$ is $K(1)$-locally conservative, we can tensor the map with $\KU$, or equivalently, take a relative tensor product over $K(X)$ with $K(X)\otimes \KU\simeq K(X[\zeta_{p^\infty}])$ where the latter equivalence comes from \cite[Theorem 1.4]{BCM}. Note, though, that this is also true for $K(C)$ for arbitrary $C$ (including $\Perf(A)$), in the sense that $K(C)\otimes \KU \simeq K(C\otimes_{\Perf(X)}\Perf(X[\zeta_{p^\infty}]))$, by \cite[Theorem 3.11]{BCM} applied to the invariant $E:= K(C\otimes_{\Perf(X)}-)$. 

Thus, we are reduced to the case where $X$ has primitive $p$ power roots of unity, so that by assumption, $A$ is split by a $p$-power Galois cover $Y\to X$. Then, $K(A)\simeq K(Y)^{hG}$ for some twisted $G$-action on $K(Y)$, by \cite[Theorem B]{CMNN}, and similarly $K(A\otimes_{\Perf(X)}C)\simeq K(\Perf(Y)\otimes_{\Perf(X)}C)^{hG}$. Thus, our map is obtained as the $G$-fixed points of an map which is an equivalence by \Cref{lm:GaloisKünneth}. 
\end{proof}
\begin{rmk}
     We shall not need this, but combining \Cref{lm:GaloisKünneth,cor:AzKünneth} with Efimov's theorem about the rigidity of motives, one can deduce that the $K(1)$-local $X$-linear motives of $\Perf(Y)$ and $\Perf(A)$ respectively are \emph{cellular}, i.e. in the localizing subcategory of $L_{K(1)}\Mot_X$ generated by the unit. This is part of a general phenomenon that $K(1)$-local motives tend to be more cellular than general motives. 
 \end{rmk}
In fact, in $K(1)$-local motives over ordinary schemes, the things that tend to prevent cellularity seem to lie in the geometry and arithmetic (the non-pointness) of those schemes. This (and sufficient optimism) motivates the following conjecture: 
\begin{conj}\label{conj:crazy}
    The category of $K(1)$-local motives relative to $\overline{\mathbb Q}$ is cellular. 
\end{conj}
It may be more reasonable to conjecture this only for dualizable $K(1)$-local motives. 

In our situation though, we cannot hope for the whole category to be cellular. However, since not all $p$-power torsion Azumaya algebras are split by $p$-power Galois covers, even in the affine case, the following does not follow from \Cref{cor:AzKünneth}, though our optimism prevails again:
\begin{conj}\label{conj:AzK}
    Let $X$ be a qcqs scheme in which $p$ is invertible, and let $A,A'$ be $p$-power torsion Azumaya algebras. The Künneth map $K(A)\otimes_{K(X)}K(A')\to K(A\otimes_X A')$ is a $K(1)$-local equivalence. 
\end{conj}

We mention here the following elementary observation, to justify that this conjecture is enough to get ``everything'' going: 
\begin{lm}\label{lm:globalKfromlocalK}
    Let $F:C\to D$ be a unital, lax symmetric monoidal functor between symmetric monoidal categories, e.g. $L_{K(1)}K(-): \Catperf_X\to \Mod_{K(X)}(\Sp_{K(1)})$. 

    Let $x\in C$ be a dualizable object (e.g. an invertible object). Suppose that the canonical map $F(x)\otimes F(x^\vee)\to F(x\otimes x^\vee)$ is an equivalence. In this case, ``Künneth holds at $x$'', that is, for every $y$, the canonical map $$F(x)\otimes F(y)\to F(x\otimes y)$$ is an equivalence.  
\end{lm}
\begin{proof}
The assumption certainly guarantees that $F(x)$ is dualizable with dual $F(x^\vee)$ and duality data coming from that of $x$ combined with the inverse of the canonical map. This allows us to define a map backwards: $$F(x\otimes y)\to F(x)\otimes F(x^\vee)\otimes F(x\otimes y)\to F(x)\otimes F(x^\vee\otimes x\otimes y)\to F(x)\otimes F(y)$$
    where the first map is the unit of the duality data, the second uses the lax symmetric monoidal structure, and the last one uses the coevaluation. 

    From there, it is a standard diagram chase to verify that this is an inverse to the canonical map, and we leave it to the reader.
\end{proof}
Thus in particular, if one can prove the Künneth formula for $A$ and $A\op$, we get it for $A$ and arbitrary $X$-linear categories $C$, as in \Cref{cor:AzKünneth}. 

Note that we stated the conjecture for all schemes, but it is perhaps more reasonable to restrict it to a nicer class of schemes, such as the ones appearing in Thomason's hyperdescent theorem, see \cite[Theorem 7.14]{CM} for somewhat optimal assumptions; or maybe to affines.

\begin{rmk}
    By a result of Gabber and de Jong \cite{deJGabber}, torsion Azumaya algebras are all represented by classical Azumaya algebras whenever $X$ has an ample line bundle. Now if $A$ is a classical Azumaya algebra of rank $r$ prime to $p$, Tabuada and Van den Bergh prove in \cite[Theorem 2.1]{TvdB1}\footnote{See also \cite[Theorem B.12]{TvdB2}.} that the motive of $A$, when $r$ is inverted (e.g. in $p$-complete motives) is trivial. Thus prime-to-$p$ torsion Azumaya algebras tend to be cellular for very different reasons than the ones outlined earlier for $p$-power torsion ones. 
\end{rmk}

So in general, we seem to be stuck. There are certain situations, though, where one can make good on the second conjecture, by verifying that the assumptions of \Cref{cor:AzKünneth} are satisfied. We give two examples below. First, the case of fields is provided by the Merkurjev--Suslin theorem \cite{MS}: 
\begin{thm}
    Let $F$ be a field of characteristic $\neq p$ having $p$-power roots of unity. Any $p$-power torsion Azumaya algebra over $F$ is (Morita equivalent to) a tensor product of cyclic Azumaya algebras, and hence is split by a $p$-power Galois extension.
\end{thm}
We also note that this was extended to semi-local rings by Hoobler \cite{hoobler}.
Combining this with \Cref{cor:AzKünneth}, we find: 
\begin{cor}\label{cor:semilocal}
    Let $F$ be a field of characteristic $\neq p$ or more generally a semilocal ring with $p$ invertible. In this case, any Azumaya algebra over $F$ which is $p$-power torsion in the Brauer group of $F$ satisfies the Künneth formula. 
\end{cor}
A second rich source of examples is the following - we thank Dustin Clausen and Akhil Mathew for pointing out that this should be true, and specifically Akhil Mathew for providing (essentially all) details and allowing me to include a proof (as usual, all mistakes and inaccuracies are due to the author): 
\begin{cor}\label{cor:perfectoid}
    Let $R$ be a $p$-henselian ring, e.g. a $p$-complete ring. In this case, any Azumaya algebra over $R[\frac{1}{p}]$ which is $p$-power torsion in the Brauer group satisfies the Künneth formula. 
\end{cor}
We first need a couple of intermediary results. First, we record the following (very special case of a) theorem of \v{C}esnavi\v{c}ius:
\begin{thm}[{\cite[Theorem 4.10]{purityBr}}]\label{purity}
Let $A$ be a $p$-torsion free integral perfectoid ring. The cohomology group $H^2_\et(\Spec(A[\frac{1}{p}]), \Z/p^\infty)$ is $0$.  
 \end{thm}
The idea is then to reduce to the perfectoid situation by adding $p$-power roots of unity and extracting $p$th roots, thus forming a pro-$p$-Galois extension. For this, we give the following criterion for the $p$-completion of a ring to be integral perfectoid, which we learned from Akhil Mathew:
\begin{prop}\label{prop:criterionperfd}
    Let $A$ be a $p$-torsion free ring. For $A^\wedge_p$ to be integral perfectoid, it suffices that 
    
\begin{enumerate}
    \item $A/p$ is semi-perfect; 
    \item There exists a unit $u$ such that $pu\in A$ admits a $p$-th root; 
    \item $A$ is $p$-integrally closed (if $x\in A[\frac{1}{p}]$ is such that $x^p\in A$, then $x\in A$). 
\end{enumerate}

\end{prop}
\begin{proof}
    We apply \cite[Lemmas 3.9 and 3.10]{BMS} with $\pi= (pu)^{\frac{1}{p}}$. 

We first verify that $\pi$ is a nonzero divisor in $A^\wedge_p$. Suppose $x_n \in A$ is such that $\pi x_n = p^n y_n$ for some $y_n \in A$. Then $ pu x_n^p = p^{np}y_n^p$ and so $x_n^p = u^{-1} p^{np-1}y_n^p$. With $n+1$ instead, we find that if $\pi x_{n+1}= p^{n+1}y_{n+1}$ for $y_{n+1}\in A$, then $(p^{-n}x_{n+1})^p\in A[\frac{1}{p}]$ is in $A$, and hence, by $p$-integral closure $x_{n+1}$ is divisible by $p^n$ in $A$. Thus if $x=(x_n)\in \lim_n A/p^n = A^\wedge_p$ has $\pi x = 0$, then $x=0$, as was to be shown.  

Furthermore, $A/\pi\to A/\pi^p$ is surjective by the first assumption, and if $x^p = \pi^p z$ in $A$, then by $p$-integral closure, $\frac{x}{\pi}\in A$ so that $x$ is divisible by $\pi$ and so this map is an isomorphism, which ultimately proves the claim. 
\end{proof}
Finally, we record an elementary calculational lemma: 
\begin{lm}\label{lm:frob}
    Let $A$ be a commutative ring, and suppose $p$ has a $p^2$th root in $A$. If every element of $A$ has a $p$th root modulo $p^{\frac{1}{p}}$, then every element of $A$ has a $p$th root modulo $p$. 
\end{lm}
\begin{proof}
    Let $x\in A$. By assumption, we may write $x= z^p + p^{\frac{1}{p}}\alpha$ for some $\alpha$. Finding a similar decomposition for $\alpha$, and using our $p^2$th root of $p$, we can write $x=z^p+p^{\frac{1}{p}}(\beta^p+\gamma p^{\frac{1}{p}}) = (z+\beta p^{\frac{1}{p^2}})^p + \gamma p^{\frac{2}{p}}\gamma$ modulo $p$. Iterating this $p$ times, we reach an expression of the form $x=a^p + p^{\frac{p}{p}} y$ modulo $p$, so $x= a^p$ modulo $p$, as was to be shown. 
\end{proof}
\begin{proof}[Proof of \Cref{cor:perfectoid}]
By \Cref{cor:AzKünneth}, it suffices to prove that any class in $H^2_\et(\Spec(R[\frac{1}{p}]),\mu_{p^\infty})$ is killed after passing to some $p$-power Galois extension of $R[\frac{1}{p},\zeta_{p^\infty}]$. So, taking the integral closure of $R$ in the latter ring, we may assume without loss of generality that $R$ has a compatible system of primitive $p$-power roots of unity. Note that taking this integral closure preserves $p$-henselianness by \cite[Tag 09XK]{stacks-project} (we will use this implicitly again below). Since over such a ring, $\mu_{p^\infty}\cong \Z/p^\infty$, we are reduced to proving the same claim for $H^2_\et(\Spec(R[\frac{1}{p}]),\Z/p^\infty)$.

We may further assume that all units in $R[\frac{1}{p}]$ have $p$-power roots: since $R[\frac{1}{p}]$ has a $p$th root of $1$, extracting $p$th roots of units yields $p$-Galois extensions, so we can extract all $p$-power roots of units in $R[\frac{1}{p}]$, and then take the integral closure of $R$ in the resulting pro-$p$-Galois extension of $R[\frac{1}{p}]$. 

Furthermore, this pro-$p$-Galois extension contains a compatible system of $p$-power roots of $p$, and hence so does the integral closure of $R$. 

All in all, we are reduced to considering the following situation: $R$ is a $p$-henselian ring containing $\Z[\zeta_{p^\infty}]$, integrally closed in $R[\frac{1}{p}]$, and every unit in the latter has a $p$th root. We claim that in this case, the $p$-completion of $R$ is integral perfectoid. 

Granting this, by \Cref{purity} (\cite[Theorem 4.10]{purityBr}) we find that $H^2_\et(\Spec(R^\wedge_p[\frac{1}{p}]), \Z/p^\infty)=0$.

By \cite[Theorem 6.11]{bhattmathew}, $H^*_\et(\Spec(R[\frac{1}{p}], \Z/p^\infty)\cong H^*_\et(\Spec(R^\wedge_p[\frac{1}{p}]), \Z/p^\infty)$, and so the claim follows for $R[\frac{1}{p}]$ as well. 

So we now need to show that $R^\wedge_p$ is integral perfectoid. Note that since we are ultimately considering $R[\frac{1}{p}]$, we may assume without loss of generality that $R$ is $p$-torsion-free. We now want to apply \Cref{prop:criterionperfd}.  

Our $R$ clearly satisfies the last two items: $R$ is integrally closed, and $p$ admits a $p$th root in $R[\frac{1}{p}]$, and hence in $R$. We now need to argue that $R/p$ is semi-perfect. 

By assumption, given $x \in R$, $1 + px$ is a unit in $R$, so it is a $p$th power in $R[\frac{1}{p}]$, and hence (by integral closure) in $R$. So let $y\in R$ be such that $(1 + px) = y^p$. It follows that $(y-1)^p$ is divisible by $p$, and so, by integral closure of $R$, that $y-1$ is divisible by $p^{\frac{1}{p}}$. So $y = 1 + p^{\frac{1}{p}} z$ for some $z \in R$.

This implies that $1 + px = 1 + p z^p +
(\textnormal{terms divisible by } p^{1 + \frac{1}{p}})$. In other words, $x = z^p+ \alpha p^{\frac{1}{p}}$ for some $\alpha$. By \Cref{lm:frob}, we are done.  
\end{proof}


\subsection{Local analysis}
\begin{nota}
    Let $E$ be a connective spectrum. We let $E[p^\infty]$ denote the fiber of the canonical map $E\to~E[\frac{1}{p}]$ in connective spectra. 
\end{nota}
\begin{defn}
    We say a Brauer class $\alpha$ on $X$ is $p$-good if it satisfies $K(1)$-local Künneth at the prime $p$, and we let $\widetilde{\BBr}(X)\subset \BBr(X)$ denote the full subspace spanned by $p$-good Brauer classes. This is clearly a sub-presheaf of $\BBr(-)$.
\end{defn}

Since strictly henselian local rings have no Brauer groups, the inclusion $\widetilde{\BBr}\to \BBr$ is an equivalence on stalks, and hence upon sheafification. 
\begin{cons}\label{cons:pshmap}
    By \Cref{cor:AzKünneth}, we obtain a map $\widetilde{\BBr}\to \PPic(L_{K(1)}K(-))$ of étale presheaves on $X$ where, by the latter, we mean the $K(1)$-local Picard group\footnote{That is, we consider modules that are $K(1)$-locally invertible.}.
\end{cons}
\begin{rmk}
    The general Künneth formula, that is, with ``arbitrary'' $X$-linear category $C$ is needed to construct this as a map of presheaves rather than just a pointwise map. Namely we need to verify that for $U\to X$ an étale map, $K(A)\otimes_{K(X)}K(U)\to K(A_{\mid U})$ is a $K(1)$-local equivalence. But by \Cref{lm:globalKfromlocalK}, a ``local'' Künneth formula involving $A$ and $A\op$ is sufficient to obtain this general Künneth formula, so if we want a subpresheaf of groups/spectra, we might as well ask for full Künneth. 
\end{rmk}

We now consider $\widetilde{\BBr}[p^\infty]$ as a subspace of $\BBr[p^\infty]\simeq \Gamma_\et(X,B^2\mu_{p^\infty})$ (cf. our convention, \Cref{conv:Br}) which, as is visible on the RHS, has a canonical $\Z$-module structure. 

Hence, the map $\widetilde{\BBr}\to \PPic(L_{K(1)}K(-))$ induces a map $$\widetilde{\BBr}[p^\infty]\to \Gpic(L_{K(1)}K(-))[p^\infty]$$ of étale presheaves on $X$. We have the following diagram of presheaves: 

\[\begin{tikzcd}
	{B^2\mu_{p^\infty}} & {\widetilde{\BBr}[p^\infty]} & {\Gpic(L_{K(1)}K(-))[p^\infty]} \\
	& {\BBr[p^\infty]}
	\arrow[from=1-1, to=1-2]
	\arrow[from=1-2, to=1-3]
	\arrow[from=1-2, to=2-2]
	\arrow["{\textnormal{Künneth formula, \Cref{conj:AzK}}}"', dashed, from=2-2, to=1-3]
\end{tikzcd}\]
where the bottom term and the top leftmost terms are the only ones known to be sheaves ahead of time, and the dashed arrow exists for $X$ whenever $X$ satisfies the conclusion of \Cref{conj:AzK}. All maps in this diagram induce equivalences on stalks and therefore on sheafification (since they are all truncated, there are no hypercompletion subtleties).

We now introduce a final gadget: 
\begin{nota}
    Let $(\Gpic^K)^a$ denote the étale sheafification of $$U\mapsto\Gpic(L_{K(1)}K(U)\otimes\Sph_{W(\overline{\mathbb F}_p)})[p^\infty] $$
\end{nota}
Here, recall that $\Sph_{W(\overline{\mathbb F}_p)}$ denotes Lurie's spherical Witt vectors, cf. \cite[Section 2.1]{ChroNS} (or see Lurie's original account \cite[Section 5.2]{Ell2}). Recall also that their goal here is to get the correct local picture, that is, to get the correct stalks.

We now state the following version of \Cref{thm:BrwithK}: 
\begin{thm}\label{thm:generalblahG}
    Let $X$ be a qcqs scheme with $p\in\Gm(X)$. There is a natural commutative diagram of the form: 

    \[\begin{tikzcd}
	{\widetilde{\BBr}(X)[p^\infty]} & {\Gpic(L_{K(1)}K(X)\otimes\Sph_{W(\overline{\mathbb F}_p)})[p^\infty]} \\
	{\BBr(X)[p^\infty]} & {(\Gpic^K)[p^\infty]^{a}(X)} 
	\arrow[from=1-1, to=1-2]
	\arrow[from=1-1, to=2-1]
	\arrow[from=1-2, to=2-2]
	\arrow[from=2-1, to=2-2]
	\end{tikzcd}\]
where: 
\begin{itemize}
    \item All the maps are inclusions of components;
    \item The bottom horizontal map is an equivalence. 
\end{itemize} 
\end{thm}
As a direct corollary, obtained by looping the above diagram, we get:
\begin{cor}
    Let $X$ be a qcqs scheme with $p\in\Gm(X)$. The map $$\PPic(X)[p^\infty]\to \Gm(L_{K(1)}K(X)\otimes\Sph_{W(\overline{\mathbb F}_p)})[p^\infty]$$ sending a line bundle to its $K$-theory class is an equivalence. 
\end{cor}
\begin{rmk}
    Not all derived line bundles are strict elements of $\PPic$, as $\Sigma \O_X$ shows when $2\neq 0$, but the torsion ones are. More generally, the subspace $\Gamma(X_\et,B\Gm)\subset \PPic(X)$ consists of invertible elements that have canonical strict structures since it clearly has a $\Z$-module structure. 
\end{rmk}
As a corollary, we also get: 
\begin{cor}
    Let $X$ be a qcqs scheme with $p\in\Gm(X)$. Suppose $X$ satisfies the conclusion of \Cref{conj:AzK}, e.g. $X$ is the spectrum of a semilocal ring. In this case, the induced map $\BBr(X)[p^\infty]\to \Gpic(L_{K(1)}K(X)\otimes\Sph_{W(\overline{\mathbb F}_p)})[p^\infty]$ is an equivalence\footnote{In particular, the latter is equivalent to the global sections of its sheafified variant.}.
\end{cor}

For the proof, we will need an elementary fact about sheafification. For this, we define: 
\begin{defn}
    Let $T$ be a Grothendieck site. A presheaf $F:T\op\to \Ss$ is called a separated presheaf if for every covering sieve $S\subset T_{/c}$, the map $F(c)\to \lim_S F(d)$ is an inclusion of components, i.e. $(-1)$-truncated (as opposed to an equivalence). 
\end{defn}
The key fact we will need about separated presheaves is the following:
\begin{lm}\label{lm:subsheaf}
    Let $T$ be a subcanonical Grothendieck site and $F:T\op\to\Ss$ be a separated presheaf, with sheafification $F^a$. The map $F\to F^a$ is section-wise an inclusion of components. 
\end{lm}
\begin{proof}
Let $x,y\in F(c)$ for some $c\in T$, classified by two maps $y(c)\to F$. The pullback $y(c)\times_F y(c)$ is a sheaf since $F$ is a separated presheaf, and hence equivalent to its own sheafification. 

Since sheafification is left exact, it follows that $F\to F^a$ induces an equivalence $y(c)\times_F y(c)\to y(c)\times_{F^a} y(c)$. Evaluating at $c$ and taking fibers over $\id_c$ implies that $\Omega(F(c),x,y)\to\Omega(F^a(c),x,y)$ is an equivalence, as was to be shown. 
\end{proof}
\begin{ex}\label{ex:Picsub}
    Let $C$ be a symmetric monoidal category with limits. For every limit diagram $R=\lim_I R_i$, the map $\PPic(\Mod_R(C))\to \lim_I \PPic(\Mod_{R_i}(C))$ is an inclusion of components. Indeed, for every invertible $L$, $L\simeq \lim_i (L\otimes_R R_i)$ (this is true more generally for dualizable modules, and the statement is as well). 
\end{ex}
\begin{cor}\label{cor:Gmsubsheaf}
    Let $X$ be a qcqs scheme with $p\in\Gm(X)$. The functor $$U\mapsto \Gpic(L_{K(1)}K(U)\otimes\Sph_{W(\overline{\mathbb F}_p)})$$ is an étale separated presheaf. 
\end{cor}
\begin{proof}
We note that $C\mapsto L_{K(1)}K(C)\otimes\Sph_{W(\overline{\mathbb F}_p)}$ is a $K(1)$-local localizing invariant and hence its restriction along $U\mapsto \Perf(U)$ is an étale sheaf by \cite[Theorem 5.39]{CM}. The above example then guarantees that composing it with $\PPic$ and then $\Map(\Z,-)$ preserves the property of being a separated presheaf. 
\end{proof}

\begin{proof}[Proof of \Cref{thm:generalblahG}]
    Consider the étale presheaves $\widetilde{\BBr}[p^\infty]$ and $$\Gpic^K: U\mapsto \Gpic(L_{K(1)}K(U)\otimes \Sph_{W(\overline{\mathbb F}_p)})[p^\infty].$$
\Cref{cons:pshmap} and naturality of the map from a presheaf to its sheafification gives us a commutative diagram as follows, where $(-)^a$ denotes sheafification: 
\[\begin{tikzcd}
	{\widetilde{\BBr}[p^\infty]} & {\Gpic^K}   \\
	{\BBr[p^\infty]} & {(\Gpic^K)^a} 
	\arrow[from=1-1, to=1-2]
	\arrow[from=1-1, to=2-1]
	\arrow[from=1-2, to=2-2]
	\arrow[from=2-1, to=2-2]
\end{tikzcd}\]
Here we use implicitly that $\widetilde{\BBr}\to \BBr$ induces an equivalence on stalks and so the latter is the sheafification of the former. 

We now claim that the following map is an equivalence on stalks: $$\widetilde{\BBr}[p^\infty]\to \Gpic^K$$

Taking this for granted, we can conclude: from the first equivalence, we get that the our square consists of maps that all become equivalences upon sheafification. Since the bottom two are sheaves, the map induces an equivalence on global sections. The leftmost vertical map is definitionally an inclusion of components, and the right vertical map is also one by combining \Cref{lm:subsheaf} and \Cref{cor:Gmsubsheaf}. 

This lets us conclude, modulo the claim about stalks, which we now prove. By subfinitariness of $\Gpic[p^\infty]$, namely \Cref{thm:finitary}, taking stalks yields a full subspace of what one gets by \emph{evaluating} at strictly henselian local rings. By Gabber rigidity, this amounts to evaluating at separably closed fields, where we get therefore a full subspace of $\Gpic(\KU_p\otimes \Sph_{W(\overline{\mathbb F}_p)})[p^\infty]$. The calculation of \cite[Theorem H]{ChroNS} shows that in this case, this space is connected, and hence the stalk is equivalent to this evaluation. 

Now the source of our map also satisfies Gabber rigidty, so the whole question is reduced to evaluating at a separably closed field $F$ of characteristic $\neq p$. There, we apply \Cref{thm:Picshift} in the case $n=0$ to conclude - note that this theorem only applies to $\Gpic(F)[p^\infty]$ (as opposed to any kind of Brauer space)  for our separably closed field $F$, and there may be issues with $\pi_0$. It is precisely for this that we adopted the convention of \Cref{conv:Br} which guarantees that $\pi_0(\BBr(F)[p^\infty]) = 0$ for a separably closed field $F$, so that we may freely check that the map is an equivalence after looping.
\end{proof}
\begin{rmk}\label{rmk:whatsneeded}
  There are other ways than \Cref{cor:AzKünneth} to obtain a Künneth formula for Azumaya algebras. For example, when $X$ is a variety over $\mathbb C$, the property of being represented by a classical Azumaya algebra and \emph{topologically trivial} in the sense of \cite[Definition 5.4]{dJPerry} (see also Lemma 5.5 in \textit{loc. cit.}) would do the trick. The proof of this uses \Cref{lm:globalKfromlocalK}, as it allows us to test Künneth with geometric objects on both sides - the geometric objects here being the associated Severi--Brauer varieties whose associated homotopy types are simply trivial projective bundles. 

  For non-topologically trivial Azumaya algebras, the question of Künneth is equivalent to an Eilenberg--Moore type question for ($p$-complete) $\KU$, and this seems a worthwhile thing to investigate. 
\end{rmk}

\subsection{$K(1)$-local $K$-theory sheaves}
In this section, we carry out the same local analysis as in the previous one, except that we are free not to worry about the Künneth formula since by taking $K$-theory \emph{sheaves} instead of $K$-theory, it comes to us for free - the followng lemma appeared without $K(1)$-localization in \cite[Proposition 6.7]{moulinos}. To state it, we introduce a notation:  
\begin{nota}
    Let $M\in\Mot_X$. We let $\mathcal K_M$ denote the étale sheaf $$Y\mapsto L_{K(1)}K(\U^X(\Perf(Y))\otimes_X M)$$ 

    For $M= \U^X(\Perf(X))$, the unit in $\Mot_X$, we simply use $\mathcal K$. 
\end{nota}
\begin{lm}\label{lm:cellularKunneth}
    Let $M,N\in\Mot_X$ be such that there is an étale cover $Y\to X$ with $M_Y$ cellular in $\Mot_Y$. In this case, the canonical map $\mathcal K_M\otimes_\mathcal K \mathcal K_N \to \mathcal K_{M \otimes N}$ is an equivalence in $\Mod_\mathcal K(\Sh(X_\et,\Sp_{K(1)}))$. 
\end{lm}
\begin{proof}
    By pullback, we see that we may assume without loss of generality that $M$ is cellular. Then the map from the presheaf tensor product is actually already an equivalence, without sheafifying.
\end{proof}

In particular, we get a map of presheaves $$\BBr(-)[p^\infty]\to \Gpic(\Mod_\mathcal K(\Sh((-)_\et, \Mod_{\Sph_{W(\overline{\mathbb F}_p)}}(\Sp_{K(1)}))))[p^\infty]$$
and to prove \Cref{thm:BrPicSh} we simply need to identify the stalks. Let us recall the statement here:
\begin{thm}\label{thm:Sh}
     Let $X$ be a qcqs scheme such that $p\in\Gm(X)$. The canonical map $$\BBr(X)[p^\infty]\to \Gpic(\Mod_{\mathcal K}(\Sh(X_\et, \Mod_{\Sph_{W(\overline{\mathbb F}_p)}}(\Sp_{K(1)}))))[p^\infty]$$ is an equivalence, where $\mathcal K$ is the étale sheaf given by $K(1)$-local $K$-theory. 
\end{thm}
\begin{proof}
   We do the same proof as before, where we ``simply'' have to identify the stalks. Now by \Cref{lm:etcpct}, the presheaf $X\mapsto \Mod_\mathcal K(\Sh(X_\et,\Mod_{\Sph_{W(\overline{\mathbb F}_p)}}(\Sp_{K(1)})))$ has values in $(\CAlg(\PrL_\st)_{(\omega)})_{\Sp_{K(1)}/}$ and so $\Gpic(-)[p^\infty]$ is subfinitary on it by \Cref{cor:catfinitary}. 

   Now by combining Gabber rigidity with general topos theory, the stalk of $U\mapsto \Mod_\mathcal K(\Sh(U_\et,\Mod_{\Sph_{W(\overline{\mathbb F}_p)}}(\Sp_{K(1)})))$ a point $\overline{x}\to X$ is simply $\Mod_{\KU_p\otimes\Sph_{W(\overline{\mathbb F}_p)}}(\Sp_{K(1)})$ whose $\Gpic[p^\infty]$ is again connected, so that for this specific diagram, subfinitariness is actual finitariness. 
Further, the map we have from $\BBr(\O_{X,\overline{x}})[p^\infty]$ is the same as earlier, and in particular an equivalence. 

   Since both source and target are truncated étale sheaves and hence hypersheaves, this stalkwise isomorphism implies an equivalence. 
\end{proof}
\Cref{thm:BrwithK} and \Cref{thm:BrPicSh} together give us the following picture of the landscape - we have a sequence of inclusions of spaces: $$\widetilde{\BBr}(X)[p^\infty]\subset \Gpic(L_{K(1)}K(X)\otimes\Sph_{W(\overline{\mathbb F}_p)})[p^\infty]\subset \Gpic(\Mod_\mathcal K(\Sh(X_\et,\Mod_{\Sph_{W(\overline{\mathbb F}_p)}}(\Sp_{K(1)}))))\simeq \BBr(X)[p^\infty]$$
Thus some $p$-power torsion Azumaya algebras can be interpreted in terms of Picard elements of $L_{K(1)}K(X)$, and all such Picard elements can be interpreted in terms of $p$-power torsion Azumaya algebras, but not necessarily in the expected way. It is tempting to believe that all inclusions in the above display are equivalences, for this would resolve this apparent oddity. 
\subsection{Forgetting the strict structure}
We start by proving \Cref{thm:nostrict}, as a ``warm-up'' to \Cref{thm:nostrictBr}. For this, we start with an easy lemma: 
\begin{lm}\label{lm:cohomology1}
 Let $\mathcal F$ be a $1$-truncated sheaf of spectra. The map $\Sigma\pi_1(\mathcal F)\to \mathcal F$ induces an injection on $\pi_0$ of global sections.
\end{lm}
\begin{proof}
We have a fiber sequence $\Omega\tau_{\leq 0}(\mathcal F)\to \Sigma \pi_1(\mathcal F)\to \mathcal F$ of sheaves. The result follows from the long exact sequence of homotopy groups obtained after applying global sections and the fact that global sections of $\Omega\tau_{\leq 0}\mathcal F$ remain coconnective. 
    \end{proof}
We are now ready to prove \Cref{thm:nostrict}, which we recall here:
\begin{thm}\label{thm:nostricttext}
    Let $X$ be a qcqs scheme such that $p\in\Gm(X)$. The canonical map $$\PPic(X)[p^\infty]\to GL_1(L_{K(1)}K(X))[p^\infty]$$ is injective on $\pi_0$.
\end{thm}
    \begin{proof}
  Consider the étale presheaf 
    $$\mathcal G: U\mapsto GL_1(L_{K(1)}K(U))[p^\infty]$$ as well as $\mathcal H := \mathcal G_{\leq 1}$. We have a map $\PPic[p^\infty]\to\mathcal G\to\mathcal H$.

    We claim that $\PPic[p^\infty]\to \mathcal H$ is a $\pi_1$-isomorphism on stalks.  Since both are $1$-truncated, \Cref{lm:cohomology1} will then give us the result. 

    For this, we note that $\pi_1(GL_1(R)[p^n])\cong \pi_1(R[p^n])$ preserves filtered colimits on $\CAlg(\Sp_{K(1)})$, and hence the stalks of the target presheaf are once again given by \emph{evaluating} at strictly henselian local rings, and hence, by Gabber rigidity, at separably closed fields. 

This means we have to fix some field $F$ of characteristic $\neq p$ and consider the map $\mu_{p^\infty}(F)\to \pi_2(\PPic(L_{K(1)}K(F))[p^\infty])\cong \coker(\pi_2(K(F)_p)\to \pi_2(K(F)_p[\frac{1}{p}]))$. That this is an isomorphism is essentially part of the statement of Snaith's theorem. Alternatively, it follows from \Cref{thm:Picshift} together with the observation that $\Gm(\KU_p)[p^\infty]\to GL_1(\KU_p)[p^\infty]$ is an isomorphism on $\pi_1$, as follows from \cite[Remark 8.18]{ChroNS}.
\end{proof}
We now prove \Cref{thm:nostrictBr}, which was: 
\begin{thm}\label{thm:nostrictBrtext}
    Let $X$ be a qcqs scheme such that $p\in\Gm(X)$. The canonical map $$\BBr(X)[p^\infty]\to \PPic(\Mod_{\mathcal K}(\Sh(X_\et, \Sp_{K(1)})))[p^\infty]$$ is injective on $\pi_0$. In particular, the restricted map $$\widetilde{\BBr}(X)[p^\infty]\to \PPic(L_{K(1)}K(X))[p^\infty]$$ is also injective on $\pi_0$.
\end{thm}
In fact, when $p$ is an odd prime, we prove the following stronger statement: 
\begin{thm}\label{thm:injoddprimes}
    Let $X$ be a qcqs scheme such that $p\in\Gm(X)$. Assume that $p$ is odd.  The canonical map $$\BBr(X)[p^\infty]\to \PPic(\Mod_{\mathcal K}(\Sh(X_\et, \Sp_{K(1)})))[p^\infty]$$ factors through a full subspace of the target, on which it has a retraction.
\end{thm}
\begin{rmk}
    It is plausible that this stronger statement is also true at the prime $2$, but our current approach only gives the injectivity statement. 
\end{rmk}
The proof of the weaker version at the prime $2$ works essentially the same as that of the stronger one, except at the end where a complication appears and results in the weaker injectivity statement. 

\begin{proof}[Proof of \Cref{thm:injoddprimes}]
    By design, our map $$\Br(X)[p^\infty]\to \PPic(\Mod_{\mathcal K}(\Sh(X_\et, \Sp_{K(1)})))[p^\infty]$$ actually lands in the locally trivial Picard spectrum of the sheaf category, $\mathbf{L}\PPic$ (this is the full subspace spanned by those sheaves that are locally equivalent to $\mathcal K$). This is the subspace on which we will construct a retraction. While we are only working at odd primes for this proof, the reader is invited to pretend at the beginning that $p$ can be $2$, since the beginning of the proof of \Cref{thm:nostrictBrtext} will be the same at that prime. 
    
We construct our retraction as follows -- the construction is more elementary and closer in spirit to the one from \cite{companionMot}. 
We have an equivalence: 
$$\mathbf{L}\PPic(\Mod_{\mathcal K}(\Sh(X_\et,\Sp_{K(1)}))\simeq \Gamma(X_\et,BGL_1(\mathcal K)).$$ We will produce a map $BGL_1(\mathcal K)[p^\infty]\to B^2\mu_{p^\infty}$ which is a retraction for the map $B^2\mu_{p^\infty}\to BGL_1(\mathcal K)$, and from there we will be done.

For this, we note that by \cite[Theorem 1.1]{CM}, the map $(K^\et)_p\to L_{K(1)}K$ is an equivalence on connective covers. It follows that it suffices to do the same for $(K^\et)_p$ in place of $\mathcal K$, and clearly it suffices to construct compatible maps $GL_1((K^\et)_p)[p^k]\to B\mu_{p^k}$ which are $\mathbb E_1$-retractions for the canonical map $B\mu_{p^k}\to GL_1((K^\et)_p)[p^k]$. We do this in the lemma below. 
\end{proof} 
\begin{lm}\label{lm:retroddprimes}
    Let $p$ be an odd prime. The canonical map $$GL_1((K^\et)_p)[p^k]\simeq \lim_n GL_1(K^\et/p^n)[p^k]\to \lim_n(GL_1(K^\et/p^n)[p^k])_{\leq 1}$$ is a retraction of the map $B\mu_{p^k}\to GL_1(K^\et/p^n)[p^k]$, in particular the target is equivalent to $B\mu_{p^k}$.
\end{lm}
\begin{proof}
    Recall that $\pi_0(K^\et/p^n)\cong \Z/p^n$ (a constant sheaf), $\pi_1(K^\et/p^n)=0$ and $\pi_2(K^\et/p^n)\cong \mu_{p^n}$ (see e.g. \cite[Corollary 6.12]{CM}). Thus $\pi_0(GL_1(K^\et/p^n)[p^k])\cong \mu_{p^k}(\Z/p^n)$ (as a constant sheaf) and $\pi_1(GL_1(K^\et/p^n)[p^k])\cong \mu_{p^n}/p^k \cong \mu_{p^k}$ for $n\geq k$. Thus we have a fiber sequence of sheaves of connective spectra/spaces $$B\mu_{p^k}\to GL_1(K^\et/p^n)[p^k]_{\leq 1}\to \mu_{p^k}(\Z/p^n)$$
    and the map $B\mu_{p^k}\to GL_1(K^\et/p^n)[p^k]_{\leq 1}$ is the canonical map (followed by the truncation map). 

    Taking limits over $n$, we are left with verifying that $\lim_n \underline{\mu_{p^k}(\Z/p^n)} = 0$ where for an abelian group $A$ we let $\underline{A}$ denote the corresponding constant étale sheaf. Note that here we are considering étale sheaves of \emph{connective} spectra, and since we are considering constant sheaves on abelian groups, we might as well consider étale sheaves of abelian groups (the inclusion of abelian groups into connective spectra preserves limits and filtered colimits). Now for an abelian group $A$, $H^0_\et(X,A) \cong C(|X|,A)$, continuous functions from the underlying topological space of $X$ into $A$. Therefore, the map $H^0_\et(X,\lim_n A_n)\to \lim_n H^0_\et(X,A_n)$ is injective for any inverse system of abelian groups $\{A_n\}$, and so the claim follows from $\mu_{p^k}(\Z_p)$ being $0$, which is true since $p$ is odd.
    \end{proof}
At the prime $2$, most of the above discussion works until the end: $\lim_n \udl{\mu_{2^k}(\Z/2^n)}$ is not zero, its global sections are $C^0(|X|,\mu_{2^k}(\Z_2))$, continuous functions from $|X|$ to $\mu_{2^k}(\Z_2)\cong \Z/2$, spanned by $\{\pm 1\}$. Thus, instead we find $\lim_n \udl{\mu_{2^k}(\Z/2^n)}\cong \udl{\Z/2}$.

We thus have to fight a bit more to prove \Cref{thm:nostrictBr} at $p=2$.
\begin{proof}[Proof of \Cref{thm:nostrictBrtext}]
    For odd primes, the result follows from \Cref{thm:injoddprimes}, so we fix $p=2$. We use ideas from the proof of \Cref{thm:injoddprimes} as well as \Cref{lm:retroddprimes}. 

    Let $B_k$ be $B(\lim_n (GL_1(K^\et/2^n)[2^k]_{\leq 1}))$. Recall that we have a sequence of maps $$B^2\mu_k\simeq \BBr(-)[2^k]\to \mathbf{L}\PPic(\Mod_{\mathcal K}(\Sh(X_\et,\Sp_{K(1)})))[2^k]\to B_k$$
    
    As in the analysis from \Cref{lm:retroddprimes}, $B_k$ is a $2$-truncated sheaf with $\pi_2\cong \mu_{2^k}$ and $\pi_1\cong \udl{\Z/2}$, and the composite map from $B^2\mu_{2^k}=\BBr(-)[2^k]$ induces an isomorphism on $\pi_2$. 

Taking $\pi_*$ of global sections of the fiber sequence $B^2\mu_{2^k}\to B_k\to B\Z/2$, we get an exact sequence of the form: $$H^0_\et(X,\Z/2)\to H^2_\et(X,\mu_{2^k})\to \pi_0(B_k(X))$$
To conclude, it thus suffices to verify that the map $H^0_\et(X,\Z/2)\to H^2_\et(X,\mu_{2^k})$ is $0$, or equivalently, to verify that $\pi_1(B_k(X))\to H^0_\et(X,\Z/2)$ is surjective. But in fact, every element in $H^0_\et(X,\Z/2)$ is in the image on $\pi_1$ of the map $BGL_1(\Sph)_X[2^k] \to B_k$, where $(-)_X$ indicates the constant étale sheaf on $X$, which proves the claim. 

\end{proof}

\section{Examples}\label{section:ex}
\subsection{Quaternions and Clifford algebras}\label{ex:quaternion}
In this example, we study the strict Picard group of $\KO$ at the prime $2$ using our results. 

On the one hand, we use our result to sketch a proof of the well-known fact (see e.g. \cite[Chapter VI, Table 3.1.1 and Theorem 3.2]{weibel})  that $L_{K(1)}K(\mathbb H) \simeq \Sigma^4\KO_2$ at the prime $2$, and in particular we produce a strict $2$-torsion structure on $\Sigma^4\KO_2$. We do this following an unpublished calculation of Ben Antieau, whom we thank for allowing us to reproduce some of the details here. As usual, all errors and inaccuracies are due to the author. 

On the other hand, after this is done, we use this calculation to obstruct the existence of a spectral lift of a certain invariant of quadratic forms.
\subsubsection{Calculation of $\Gpic$}
We work at the prime $p=2$ and consider $X=\Spec(\mathbb R)$. To simplify notation, we let $\overline{R}$ denote $R\otimes \Sph_{W(\overline{\mathbb F_2})}$ in what follows. 
\newcommand{\KKO}{\overline{\KO}}
\newcommand{\KKU}{\overline{\KU}}

 \Cref{thm:BrwithK} tells us that $$\BBr(\mathbb R)[2^\infty]\simeq \Gpic(\KKO_2)[2^\infty],$$

using that $K(\mathbb R)_2\simeq \mathrm{ko}_2$ by Suslin rigidity and so $L_{K(1)}K(\mathbb R)\simeq \KO_2$, and similarly, $$\BBr(\mathbb R)[2^\infty]^{S^1}\simeq \Gpic(\KO_2)[2^\infty]$$

We can therefore already read off the homotopy groups of both: 
\begin{cor}
    $$\Gpic(\KO_2)[2^\infty]\simeq \Z/2\oplus \Sigma\Z/2\oplus\Sigma^2\Z/2$$ and 
    $$\Gpic(\overline{\KO}_2)[2^\infty]\simeq \Z/2\oplus\Sigma^2\Z/2$$ and the map between them is an isomorphism on $\pi_0, \pi_2$. 
\end{cor}
Now, we would like to identify the relevant $\Gpic$'s without the $2$-power torsion, \emph{and} we would like to identify explicitly the $\pi_0$-classes. We begin with the $\pi_0$-classes: 
\begin{lm}\label{lm:pi0classisSigma4}
    The nonzero element in $\Gpic(\overline{\KO}_2)[2^\infty]$ (and hence in $\Gpic(\KO_2)[2^\infty]$) is a (unique up to equivalence) strict structure on $\Sigma^4\overline{\KO}_2$ (resp. $\Sigma^4\KO_2$). 
\end{lm}
\begin{proof}
    By Galois descent, we have $$\Gpic(\KKO_2)[2^\infty]\simeq~\tau_{\geq 0}(\Gpic(\KKU_2)[2^{\infty}]^{hC_2}),$$ compatibly with the forgetful maps to $\PPic$. By \cite[Theorem H]{ChroNS}, we have $\Gpic(\KKU_2)[2^\infty]\simeq~\Sigma^2\Z/2^\infty$, and the conjugation action is given by $-1$.

In particular, we see that our unique nontrivial element in $\pi_0(\BBr(\mathbb R)[2^\infty])\cong \Z/2$ is detected in the descent spectral sequence by the unique nonzero-class in $$H^2(BC_2,\pi_2\Gpic(\KKU_2)[2^\infty])\cong~H^2(BC_2, \Z/2^\infty)\cong~\Z/2.$$

Since $\pi_2\Gpic(\KKU_2)[2^\infty]\to\pi_2\PPic(\KKU_2)[2^\infty]\cong W(\overline{\mathbb F_2})\otimes\Z/2^\infty$ is injective on $H^2(BC_2,-)$, we can now try to analyze what this class corresponds to in $\pi_0\PPic(\KKO_2)[2^\infty]$. Luckily, this has been done by Gepner--Lawson in \cite{gepnerlawson}. Specifically, one sees in Figure 7.2 in \textit{loc. cit.} that the nonzero element of $H^3(BC_2,\pi_3\PPic(\KU))$ (which corresponds to our $H^2(BC_2,\pi_2\PPic(\KU)[2^\infty])$ and lives in Adams grading $(0,3)$ in that figure) survives and moreover is the ``top'' piece of a $3$-step filtration of $\Pic(\KO)\cong\Z/8$ by $\Z/2$'s, and hence it must be given by the element $4\in\Z/8$, corresponding to $\Sigma^4\KO$. 

Of course, they do this for $\KO$, but the same fact follows directly for $\KO_2$ and also $\KKO_2$, and this proves the claim.

\end{proof}
\begin{cor}
    We have equivalences 
    $L_{K(1)}K(\mathbb H)\otimes\Sph_{W(\overline{\mathbb F_2})}\simeq \Sigma^4\overline{\KKO}_2$ and $L_{K(1)}K(\mathbb H)\simeq\Sigma^4\KO_2$
\end{cor}

In total, this discussion shows: 
\begin{cor}\label{cor:GpicKO}
    $\Gpic(\KO_2)\cong\Z/2\oplus \Sigma\Z/2\oplus\Sigma^2\Z/2$, with $\pi_0$ generated by a unique strict structure on $\Sigma^4\KO_2$. 

    Similarly, $\Gpic(\overline{\KO}_2)\simeq \Z/2\oplus\Sigma\overline{\mathbb F_2}^\times \oplus \Sigma^2\Z/2$, and the map between them is an isomorphism on $\pi_0$ and $\pi_2$
\end{cor}
\begin{proof}
The abstract homotopy group calculations are direct group cohomology calculations using \cite[Theorem 8.17]{ChroNS} which computes $\Gpic(\overline{\KU}_2)$, using Hilbert's Theorem 90 to compute the $\Z$-fixed points of the Galois action on $\overline{\mathbb F}_2^\times$. Since $\Gpic$ is an Eilenberg--MacLane spectrum, this also describes the result as specta. 

The only thing to add is that \Cref{lm:pi0classisSigma4} shows that the generator of $\pi_0$ is indeed given by $\Sigma^4\KO_2$ (resp. $\Sigma^4\overline{\KO}_2$) as opposed to an abstract cohomology class. 
\end{proof}
\subsubsection{Clifford algebras}
We now exploit our results to obstruct the existence of a spectrum level ``Clifford algebra'' map from the $L$-theory of $\mathbb R$ to its Brauer--Wall spectrum. 

Recall the following construction - here, super vector space means $\Z/2$-graded vector space and super algebra means algebra in super vector spaces. Here, the tensor product is the usual graded tensor product with the Koszul sign rule for the symmetry isomorphism\footnote{The symmetry isomorphism is irrelevant for the notion of algebra, but it is relevant for the tensor product of algebras!}. 
\begin{cons}\label{cons:Cliff}
    Let $(V,q)$ be a quadratic form over the real numbers, and let $\mathrm{Cl}(V,q)$ be the free associative superalgebra on $V$ in degree $1$, modulo the relation $v^2= q(v)$ for all $v\in V$. 

    The construction $(V,q)\mapsto \mathrm{Cl}(V,q)$ is a symmetric monoidal functor on the groupoid $\mathrm{Quad}(\mathbb R)$ of quadratic forms equipped with the (orthogonal) direct sum symmetric monoidal structure, and the tensor product of algebras in the target. 

    Furthermore, $\mathrm{Cl}(V,q)$ is always an Azumaya algebra in super vector spaces and thus we obtain, by group completion, a map of connective spectra $\mathrm{GW}^s(\mathbb R)\to \mathbf{BW}(\mathbb R)$, where the source is\footnote{symmetric, to fix things, but this is irrelevant since $2\in\mathbb R^\times$.} the connective Grothendieck-Witt theory of $\mathbb R$, and the target is the Brauer--Wall spectrum of $\mathbb R$, defined as the Brauer spectrum of the category of super vector spaces. 
\end{cons}
The following is an easy observation, a special case of the more general result from \cite{knus1991clifford}: 
\begin{obs}
    Let $V$ be a real vector space, and consider the quadratic form $(V\oplus V^\vee, q)$ where $q(v,f) = f(v)$. 

    The super Azumaya algebra $\mathrm{Cl}(V\oplus V^\vee)$ is Morita equivalent to the unit. Indeed, the $\mathrm{Cl}(-)$ construction is symmetric monoidal and so one reduces to $V=\mathbb R$ where one verifies directly that $\mathrm{Cl}(\mathbb R\oplus\mathbb R^\vee) \cong \End(\mathbb R\oplus\mathbb R[1])$. 
\end{obs}
\begin{rmk}
    In fact, in general, $\mathrm{Cl}(V\oplus V^\vee)\cong \End(\Lambda^* V)$, cf. \cite[Definition 1.5]{karoubi2019real}. 
\end{rmk}
In particular, it follows that the \emph{after applying} $\pi_0$, the map of connective spectra $\mathrm{GW}^s(\mathbb R)\to \mathbf{BW}(\mathbb R)$ factors, through $\pi_0L(\mathbb R)$, aka the Witt group $W(\mathbb R)$. This was observed already by Wall when he introduced the Brauer--Wall group, cf. \cite{wall1964graded}. It is thus reasonable to ask whether this factorization also exists at the level of spectra, i.e. whether $\mathrm{GW}^s(\mathbb R)\to \mathbf{BW}(\mathbb R)$ factors through $L(\mathbb R)$. We now explain why this cannot be the case. First, we note that such a factorization would provide us with many strict elements in the Brauer--Wall spectrum: 
\begin{lm}
The $2$-localization $L(\mathbb R)_{(2)}$ is an Eilenberg--MacLane spectrum, and hence, any map $L(\mathbb R)\to \mathbf{BW}(\mathbb R)$ lifts to a map to $\Map(\mathbb Z,\mathbf{BW}(\mathbb R))$. 
\end{lm}
\begin{proof}
    The first part is \cite[Proposition 6.3]{land2018relation}. The second part follows since $\mathbf{BW}(\mathbb R)$ is $2$-local: its homotopy groups are $\Z/8$ in degree $0$, $\Z/2$ in degree $1$ and $\mathbb R^\times\cong \Z/2\times \mathbb R$ in degree $2$. 
\end{proof}
\begin{cor}\label{cor:iffactthenstr}
    If the map $\mathrm{GW}^s(\mathbb R)\to \mathbf{BW}(\mathbb R)$ factors through $L(\mathbb R)$, then every element in $\mathbf{BW}(\mathbb R)$ admits a strict structure.

    More generally, this is so if the map $\pi_0L(\mathbb R)\to \pi_0\mathbf{BW}(\mathbb R)$ admits a lift to a map of spectra. 
\end{cor}
\begin{proof}
    Indeed, the $\pi_0$-map in question is surjective by the standard calculation of $\mathrm{BW}(\mathbb R)$. 

    Therefore, the claim follows from the previous lemma. 
\end{proof}

Now there is a symmetric monoidal functor from super vector spaces to $\mathbb R[u^{\pm 1}]$-modules, where $|u|=2$. We therefore obtain a map of spectra $\mathbf{BW}(\mathbb R)\to \BBr(\mathbb R[u^{\pm 1}])$ (in fact, this map is an isomorphism on $2$-truncations but we will not need this). 

Note that $\mathrm{BW}(\mathbb R)\cong \mathbb Z/8$, whereas $\mathrm{BW}(\mathbb C)\cong \mathbb Z/2$, and the map from the former to the latter is surjective, and thus it kills $\mathbb Z/4\subset \mathbb Z/8$. We thus obtain: 
\begin{prop}\label{prop:superK}
    Let $A$ be a super Azumaya algebra representative of any element of $\mathbb Z/4\subset \mathbb Z/8\cong \mathrm{BW}(\mathbb R)$. 

    For any $\mathbb R[u^{\pm 1}]$-linear category $C$, we have a Künneth formula in $K(1)$-local $K$-theory at the prime $2$: $K(A)\otimes_{K(\mathbb R[u^{\pm 1}])}K(C)\to K(A\otimes_{\mathbb R[u^{\pm 1}]}C)$ is a $K(1)$-local equivalence. 

    Therefore, on the union of connected components $\widetilde{\mathbf{BW}}(\mathbb R)$ corresponding to $\Z/4$, $K(1)$-local $K$-theory defines a map of spectra $\widetilde{\mathbf{BW}}(\mathbb R)\to \PPic(L_{K(1)}K(\mathbb R[u^{\pm 1}]))$.  
\end{prop}
\begin{proof}
    By the discussion preceding the proposition, any such $A$ is trivialized by the $C_2$-Galois extension $\mathbb C[u^{\pm 1}]/\mathbb R[u^{\pm 1}]$, and thus the claim follows from (the same proof as in) \Cref{lm:GaloisKünneth}. 
\end{proof}

The idea now is that topological $K$-theory takes the $\Z/8$ appearing in the Brauer--Wall group of $\mathbb R$ to the $\Z/8$ in $\Pic(\KO)$. Therefore, the above proposition seems to produce a strict structure on the generator of $\Z/4$, namely $\Sigma^2\KO$. But our calculation of $\Gpic(\KO_2)$ shows that there is no such thing. We could be done here and there by actually using topological $K$-theory of graded $C^*$-algebras, but instead we go a slightly more complicated route to stick to $K(1)$-local $K$-theory. 

Thus, let us now make the above idea precise in that context. 

\begin{obs}
    The quaternion algebra $\mathbb H$, viewed as a superalgebra in grading $0$, is the element $4\in\mathbb Z/8\cong \mathrm{BW}(\mathbb R)$.

As an element in $\BBr(\mathbb R[u^{\pm 1}])$, it is basechanged from $\mathbb R$. 
\end{obs}
Our goal is now to observe that, at least after basechange to $\Sph_{W(\overline{\mathbb F}_2)}$, the map $K(\mathbb R)\to K(\mathbb R[u^{\pm 1}])$ has a $K(1)$-local splitting, which will produce a direct contradiction since in $\mathbb Z/4\subset \mathrm{BW}(\mathbb R)$, $\mathbb H$ is divisible by $2$, but in $\Gpic(\overline{\KO}_2)$, $L_{K(1)}K(\mathbb H)\otimes\Sph_{W(\overline{\mathbb F}_2)}$ is not. 

 To explain this $K(1)$-local splitting, we first recall: 
\begin{cons}\label{cons:strSigma}
Since $\mathbb Z\cong \pi_0\PPic(\mathbb Z)\to\pi_1\PPic(\mathbb Z)\cong \mathbb Z^\times$ is the unique surjection, \cite[Proposition 3.23]{Cyclochro} shows that $\Sigma^2\mathbb Z$ has a canonical strict structure. It thus induces specific strict structure on $1=[\Sigma^2\mathbb Z]\in K(\mathbb Z)$. 
\end{cons}

\begin{prop}\label{prop:isthom}
    After completion at any prime, $K(\mathbb R[u^{\pm 1}])$ is an $\mathbb E_\infty$-Thom spectrum over $K(\mathbb R)$ associated with the map $B\mathbb Z\to BGL_1(K(\mathbb R))\subset \PPic(K(\mathbb R))$ given by the above construction. 
\end{prop}
\begin{proof}
    Before taking $K$-theory, this is the case at the category level by \cite[Theorem 7.13]{CCRY}: $\Perf(\mathbb R[u^{\pm 1}])$ is a Thom object over $\Perf(\mathbb R)$ associated to the $\mathbb E_\infty$ map $B\mathbb Z\to B\PPic(\mathbb R)$ classifying $\Sigma^2\mathbb R$ as a strict picard object. 

    Thus, after applying $K$-theory we obtain a canonical comparison of $\mathbb E_\infty-K(\mathbb R)$-algebras $$\colim_{B\mathbb Z}K(\mathbb R)\to K(\mathbb R[u^{\pm 1}])$$ from the Thom spectrum. One can then ask whether this is an equivalence, and for this one may simply forget the multiplicative structure. In fact, it is not an equivalence, but becomes so after $p$-completing for any $p$. 

    This can be proved for example by using the fiber sequence $K(\mathbb R)\to K(\mathbb R[u])\to  K(\mathbb R[u^{\pm 1}])$ coming from the theorem of the heart, and using Goodwillie's theorem to prove that $K(\mathbb R[u])\to K(\mathbb R)$ is a mod $p$ equivalence. 
\end{proof}
It is thus an interesting question to know whether this strict structure on $[\Sigma^2 \mathbb R]$ is trivial, since that would make the relevant Thom object an ordinary group ring. We claim that at least its image in $\Gm(\overline{\KO}_2)$ is trivial. 
\begin{lm}
$\Gm(\overline{\KO}_2)\simeq \overline{\mathbb F}_2^\times \oplus \Sigma\mathbb Z/2 $ and the map to $GL_1$ is given by the multiplicative lifts $\overline{\mathbb F}_2^\times\to W(\overline{\mathbb F}_2)^\times$ on $\pi_0$, and in particular is injective on $\pi_0$. 
\end{lm}
\begin{proof}
    This follows from \Cref{cor:GpicKO}. 
\end{proof}

\begin{cor}\label{cor:strtriv}
    The image of the strict element $[\Sigma^2\mathbb Z]\in \Gm(K(\mathbb Z))$ from \Cref{cons:strSigma} is trivial in $\Gm(\overline{\KO}_2)$. 
\end{cor}
\begin{proof}
    The previous Lemma shows that $\Gm(\overline{\KO}_2)\to GL_1(\overline{\KO}_2)$ is injective on $\pi_0$. Since $[\Sigma^2\Z]\in GL_1(K(\mathbb Z))$ is already equivalent to $1$, the claim follows. 
\end{proof}
The following result is probably not optimal, but suffices for us:
\begin{cor}
    After $2$-completion and tensoring with $\Sph_{W(\overline{\mathbb F}_2)}$, the map $K(\mathbb R)\to K(\mathbb R[u^{\pm 1}])$ witnesses the target as a group ring over the source, and in particular admits an $\mathbb E_\infty$-splitting. 
\end{cor}
\begin{proof}
    This follows by combining \Cref{prop:isthom} and \Cref{cor:strtriv}. 
\end{proof}
\begin{cor}\label{cor:nofact}
    The map $W(\mathbb R)\to \mathrm{BW}(\mathbb R)$ does not lift to a map of spectra $L(\mathbb R)\to \mathbf{BW}(\mathbb{R})$. 
\end{cor}
\begin{proof}
By \Cref{prop:superK}, the map $A\mapsto L_{K(1)}K(A)$ at the prime $2$ defines a map of spectra from $\widetilde{\mathbf{BW}}(\mathbb R)$ to $\PPic(L_{K(1)}K(\mathbb R[u^{\pm 1}]))$, where the source is the union of components of $\mathbf{BW}(\mathbb R)$ corresponding to $\mathbb Z/4\subset \mathrm{BW}(\mathbb R)$. 

 Fixing any $\mathbb E_\infty$-splitting of the map $K(\mathbb R)\to K(\mathbb R[u^{\pm 1}])$ after $K(1)$-localizing and tensoring with $\Sph_{W(\overline{\mathbb F}_2)}$, we get a map $\widetilde{\mathbf{BW}}(\mathbb R)\to \PPic(\overline{\KO}_2)$ such that the composite $\BBr(\mathbb R)\to \widetilde{\mathbf{BW}}(\mathbb R)\to \PPic(\overline{\KO}_2)$ is our standard map. 

 Now consider the element $1\in\mathbb Z/4\subset \mathrm{BW}(\mathbb R)$. On the one hand, it doubles to the class of $\mathbb H$ and so is sent to $2\in\mathbb Z/8\subset \Pic(\overline{\KO_2})$. 
 
On the other hand, if there exists a lift to a map of spectra of the Clifford algebra assignment, then, by \Cref{cor:iffactthenstr}, this element admits a strict structure, and therefore so does $\Sigma^2\overline{\KO}_2$. By \Cref{cor:GpicKO}, this is not the case, and so there exists no such factorization. 
\end{proof}
\subsection{Number rings}\label{section:numberring}
In this section, we use our results together with the Lichtenbaum-Quillen conjecture (a theorem of Voevodsky and Rost) to access $\Gm(K(R))[p^\infty]$ for certain commutative rings $R$, namely those  ``weakly of Lichtenbaum-Quillen dimension $\leq 2$''.

\begin{defn}
    Let $X$ be a qcqs scheme. We say $X$ is weakly of Lichtenbaum-Quillen (LQ) dimension $\leq d$ at the prime $p$ if the cofiber of the map $K(X)_p\to L_{K(1)}K(X)$ is in degrees $\leq d-2$. 
\end{defn}
\begin{rmk}
    In \cite{EldenLQ}, the authors introduce a related notion of Lichtenbaum-Quillen dimension, and the above definition was stolen from theirs. Their notion differs from ours in two ways: first, they look at all primes at the same time; and second, they look at a more refined version of our condition (which is why we appended the word ``weakly'') since instead of comparing directly $K(X)_p \to L_{K(1)}K(X) = L_{T(1)}K(X)$, they require explicit connectivity estimates for the Bott map $\Sigma^2 K(X)/p^r\to K(X)/p^r$ which one inverts to obtain $L_{T(1)}K(X)/p^r$.
\end{rmk}
\begin{ex}\label{ex:ringsofintegers}
By class field theory, number fields are of virtual mod $p$ cohomological dimension $\leq 2$ for all primes $p$, and finite fields are of mod $p$ cohomological dimension $1$ for all primes $p$. Combining this with \cite[Theorems 1.1 and 1.2]{CM} we find that if $X$ is the spectrum of a ring of $S$-integers $\O_F[\frac{1}{p}, p\in S]$ for some number field $F$, or the spectrum of a finite field, then $X$ is weakly of LQ dimension $\leq 2$ at all primes $p$ invertible in $X$. 
\end{ex}
One can of course make this definition for a stable category as well, but beyond those coming from algebraic geometry, it does not seem like we have access to anything like a Lichtenbaum-Quillen conjecture, and thus to an understandable LQ dimension (see also \cite[Section 1.2]{EldenLQ}). 

\begin{lm}
    Let $X$ be a qcqs scheme in which $p$ is invertible, and weakly of LQ dimension $\leq 2$ at the prime $p$. In this case the following two maps are equivalences: $$\PPic(X)[p^\infty]\to \Gm(K(X)_p\otimes\Sph_{W(\overline{\mathbb F}_p)})[p^\infty]\to \Gm(L_{K(1)}K(X)\otimes\Sph_{W(\overline{\mathbb F}_p)})[p^\infty]$$
\end{lm}
\begin{proof}
    We already know that the total composite is an equivalence by \Cref{thm:BrwithK}, so it suffices to observe that $K(X)_p\otimes\Sph_{W(\overline{\mathbb F}_p)}\to L_{K(1)}K(X)\otimes\Sph_{W(\overline{\mathbb F}_p)}$ is an inclusion of components because this is the case by assumption without $\Sph_{W(\overline{\mathbb F}_p)}$, and the latter is $p$-completely a direct sum of copies of $\Sph_p$. Now taking $\Gm(-)[p^\infty]$ remains an inclusion of components, and since it factors an equivalence, it must be an equivalence too. 
\end{proof}
\begin{cor}\label{cor:GmKp}
     Let $X$ be a qcqs scheme in which $p$ is invertible, and weakly of LQ dimension $\leq 2$ at the prime $p$. There is an equivalence $$\Gm(K(X)_p)[p^\infty]\simeq \PPic(X)[p^\infty]^{S^1}$$
\end{cor}
\begin{proof}
    The map $\PPic(X)[p^\infty]\to \Gm(K(X)_p\otimes \Sph_{W(\overline{\mathbb F}_p)})[p^\infty]$ is clearly $\widehat{\Z}$-equivariant for the trivial action on the source, and the Galois action on the target. Taking $\Z$-fixed points, using \Cref{rmk:ZactSphW}, gives the result.
\end{proof}
\begin{lm}\label{lm:primetop}
    Let $R$ be a commutative ring spectrum. The forgetful maps $\Gm(R[\frac{1}{p}])[p^\infty]\to GL_1(R[\frac{1}{p}])[p^\infty]\to GL_1(\pi_0(R)[\frac{1}{p}])[p^\infty]$ are equivalences. 
\end{lm}
\begin{proof}
    $p$ is invertible in the fiber of $GL_1(R[\frac{1}{p}])\to GL_1(\pi_0(R)[\frac{1}{p}])$, so that mapping in from $\Z/p^n$ produces an equivalence for every $n$, and thus also in the colimit. This proves the result about the second map. But now the target of the second map is in degree $0$, i.e. an ordinary abelian group, so that taking the (connective!) mapping spectrum from $\Z$ leaves it unchanged, which proves the statement. 
\end{proof}

Our main example is then as follows: 
\begin{cor}\label{cor:GmLQ}
     Let $X$ be a noetherian scheme of Krull dimension $\leq 1$ in which $p$ is invertible. Suppose that the underived Picard group $\Pic^\heartsuit(X)$ is finite and $X$ is weakly of LQ dimension $\leq 2$ at the prime $p$. Suppose finally that $K_{-1}(X)=0$.

     In this case, there is an equivalence $$\Gm(K(X))[p^\infty]\simeq \PPic(X)[p^\infty]\oplus \Gm(X)[p^\infty]$$
\end{cor}
\begin{ex}
    By combining \Cref{ex:ringsofintegers} with finiteness of the class number for rings of $S$-integers in number fields as well as vanishing of negative $K$-theory for those (e.g. by regularity), we find that if $F$ is a number field, $X=\Spec(F)$ and $\Spec(\O_F[\frac{1}{p}])$ satisfy these assumptions. 
\end{ex}
\begin{ex}
    By \Cref{ex:ringsofintegers}, these assumptions are also satisfied by finite fields, and so we recover the calculation of the prime-to-characteristic torsion part of \cite[Theorem 6.2.1]{shacharkiran}. For the non-torsion part (and the at-the-characteristic torsion part), we simply note that all the higher homotopy groups of $K(F)$, for a finite field $F$, are prime-to-characteristic torsion, and hence $\Map(\Z_{(\mathrm{char}(F))}, GL_1(K(F)))= \Map(\Z_{(\mathrm{char}(F))}, GL_1(K_0(F))=  \hom(\Z_{(\mathrm{char}(F))}, \Z/2)$, so we recover the whole calculation this way. 
\end{ex}
\begin{proof}
    Since $X$ is noetherian of dimension $1$, we have by \cite[Corollary II.2.6.3]{weibel} that $K_0(X) \cong H^0(X,\Z)\oplus\Pic^\heartsuit(X)$. It follows from finiteness of $\Pic^\heartsuit(X)$ and vanishing of $K_{-1}(X)$ that the map $K_0(X)\to \pi_0(K(X)_p)$ is simply the $p$-completion map $$H^0(X,\Z)\oplus\Pic^\heartsuit(X)\to H^0(X,\Z_p)\oplus\Pic^\heartsuit(X)_p $$
Upon inverting $p$, it becomes the map $$H^0(X,\Z[\frac{1}{p}])\oplus\Pic^\heartsuit(X)_{p'}\to H^0(X,\mathbb Q_p)$$
where for a finite abelian group $A$, $A_{p'}$ denotes the prime-to-$p$ part of $A$. Note that by \cite[Corollary II.2.6.2]{weibel}, the source of this morphism is a trivial square zero-extension of $H^0(X,\Z[\frac{1}{p}])$ by $\Pic^\heartsuit(X)_{p'}$, so that upon taking $p$-power roots of unity, we get the map $$H^0(X,\mu_{p^\infty}(\Z[\frac{1}{p}]))\to H^0(X,\mu_{p^\infty}(\mathbb Q_p))$$ which is an isomorphism. 

Combining this with \Cref{lm:primetop}, we find that $\Gm(K(X)[\frac{1}{p}])[p^\infty]\to \Gm(K(X)_p[\frac{1}{p}])[p^\infty]$ is an equivalence (of abelian groups). 

Using the arithmetic fracture square: 
\[\begin{tikzcd}
	{K(X)} & {K(X)[\frac{1}{p}]} \\
	{K(X)_p} & {K(X)_p[\frac{1}{p}]}
	\arrow[from=1-1, to=1-2]
	\arrow[from=1-1, to=2-1]
	\arrow[from=1-2, to=2-2]
	\arrow[from=2-1, to=2-2]
\end{tikzcd}\]
we therefore find that $\Gm(K(X))[p^\infty]\to \Gm(K(X)_p)[p^\infty]$ is an equivalence too, so we conclude with \Cref{cor:GmKp}. 
    
\end{proof}

\bibliographystyle{alpha}
\bibliography{Biblio.bib}

\end{document}